\documentclass[11pt,reqno]{amsart}
\usepackage{amsfonts}
\usepackage[margin=1in]{geometry}

\usepackage{amsmath,amsfonts,amssymb,amsthm}
\usepackage[abbrev,lite,nobysame]{amsrefs}
\usepackage{mathrsfs,esint}
 \usepackage[usenames,dvipsnames]{color}
\usepackage{dsfont}
\usepackage{appendix}
\usepackage{enumerate,enumitem}
\usepackage{comment}
\usepackage{bbm}
\usepackage{xfrac}
\usepackage{graphicx}
\usepackage{array}

\usepackage{tikz}
\usetikzlibrary{decorations.pathreplacing} 
\usetikzlibrary{arrows.meta}
\usetikzlibrary{calc}

\usepackage[colorlinks=true, pdfstartview=FitV, linkcolor=BrickRed,citecolor=black, urlcolor=black]{hyperref}

\usepackage{mathtools}
\mathtoolsset{showonlyrefs}

\usepackage[dvipsnames,table,svgnames]{xcolor}

\newtheorem{proposition}{Proposition}[section]
\newtheorem{lemma}[proposition]{Lemma}

\newtheorem{theorem}[proposition]{Theorem}

\theoremstyle{definition}
\newtheorem{definition}[proposition]{Definition}

\newtheorem{remark}[proposition]{Remark}

\numberwithin{equation}{section}

\newcommand{\blue}[1]{\textcolor{blue}{#1}}
\newcommand{\T}{\mathbb{T}}
\newcommand{\Dt}{\Delta t}

 \makeatletter

\renewcommand\subsubsection{\@startsection{subsubsection}{3}%
\normalparindent{.5\linespacing\@plus.7\linespacing}{-.5em}
{\normalfont\bfseries}}

\def\@tocline#1#2#3#4#5#6#7{\relax
  \ifnum #1>\c@tocdepth 
  \else
    \par \addpenalty\@secpenalty\addvspace{#2}%
    \begingroup \hyphenpenalty\@M
    \@ifempty{#4}{%
      \@tempdima\csname r@tocindent\number#1\endcsname\relax
    }{%
      \@tempdima#4\relax
    }%
    \parindent\z@ \leftskip#3\relax \advance\leftskip\@tempdima\relax
    \rightskip\@pnumwidth plus4em \parfillskip-\@pnumwidth
    #5\leavevmode\hskip-\@tempdima
      \ifcase #1
       \or\or \hskip 1em \or \hskip 2em \else \hskip 3em \fi%
      #6\nobreak\relax
    \dotfill\hbox to\@pnumwidth{\@tocpagenum{#7}}\par
    \nobreak
    \endgroup
  \fi}
\makeatother

\newcommand\eps{\varepsilon}
\newcommand\e{{\rm e}}
\newcommand\dd{{\rm d}}
\newcommand\Ker{{\rm Ker}}
\newcommand\Ran{{\rm Ran}}
\newcommand\Id{{\rm Id}}
\def\Re{{\rm Re}}
\newcommand{\dist}{{\rm dist}}

\newcommand{\RR}{\mathbb{R}}
\newcommand{\NN}{\mathbb{N}}
\newcommand{\CC}{\mathbb{C}}
\newcommand{\TT}{\mathbb{T}}

\newcommand{\jj}{\boldsymbol{j}}

\newcommand{\cL}{\mathcal{L}}

\newcommand{\cS}{\mathcal{S}}

\newcommand{\initial}{\text{in}}
\newcommand{\xx}{\mathbf{x}}

\author[M. Sorella]{Massimo Sorella}
\address{Department of Mathematics, Imperial College London}
\email{m.sorella@imperial.ac.uk}

\author[D. Villringer]{David Villringer}
\email{d.villringer22@imperial.ac.uk}

\title[A limsup fast dynamo on $\TT^3$]{A limsup fast dynamo on $\TT^3$} 

\begin{document}


\begin{abstract}
We construct a time-dependent, incompressible, and uniformly-in-time Lipschitz continuous velocity field on $\TT^3$ that produces exponential growth of the magnetic energy along a subsequence of times, for every positive value of the magnetic diffusivity. Because this growth is not uniform in time but occurs only along a diverging sequence of times, we refer to the resulting mechanism as a limsup fast dynamo.
Our construction is based on suitably rescaled Arnold–Beltrami–Childress (ABC) flows, each supported on long time intervals. The analysis employs perturbation theory to establish continuity of the exponential growth rate with respect to both the initial data and the diffusivity parameter.
This proves the weak form of the fast dynamo conjecture formulated by Childress and Gilbert in  \cite{CG95} on $\TT^3$, but the considerably more challenging version proposed by Arnold \cite{Arnold04} on $\TT^3$ remains an open problem.
\end{abstract}

\maketitle

\tableofcontents

\section{Introduction}

The kinematic dynamo problem seeks to explain how magnetic fields in astrophysical and geophysical systems, such as stars and planets, can be sustained and amplified.
In these systems, magnetic fields are stretched, twisted, and folded by fluid motion, leading to their growth despite resistive effects. 
In 1919, Larmor~\cite{larmor1919possible} was the first to propose that the magnetic fields of the Sun and Earth are generated and sustained by the motion of electrically conducting fluids within these celestial bodies.  
A concrete observable example is given by sunspots, temporary dark regions on the Sun’s surface, see Figure \ref{fig:moffatt}, that are linked to a change of magnetic field intensity. 
\begin{figure}[h]
    \centering
    \includegraphics[width=6cm]{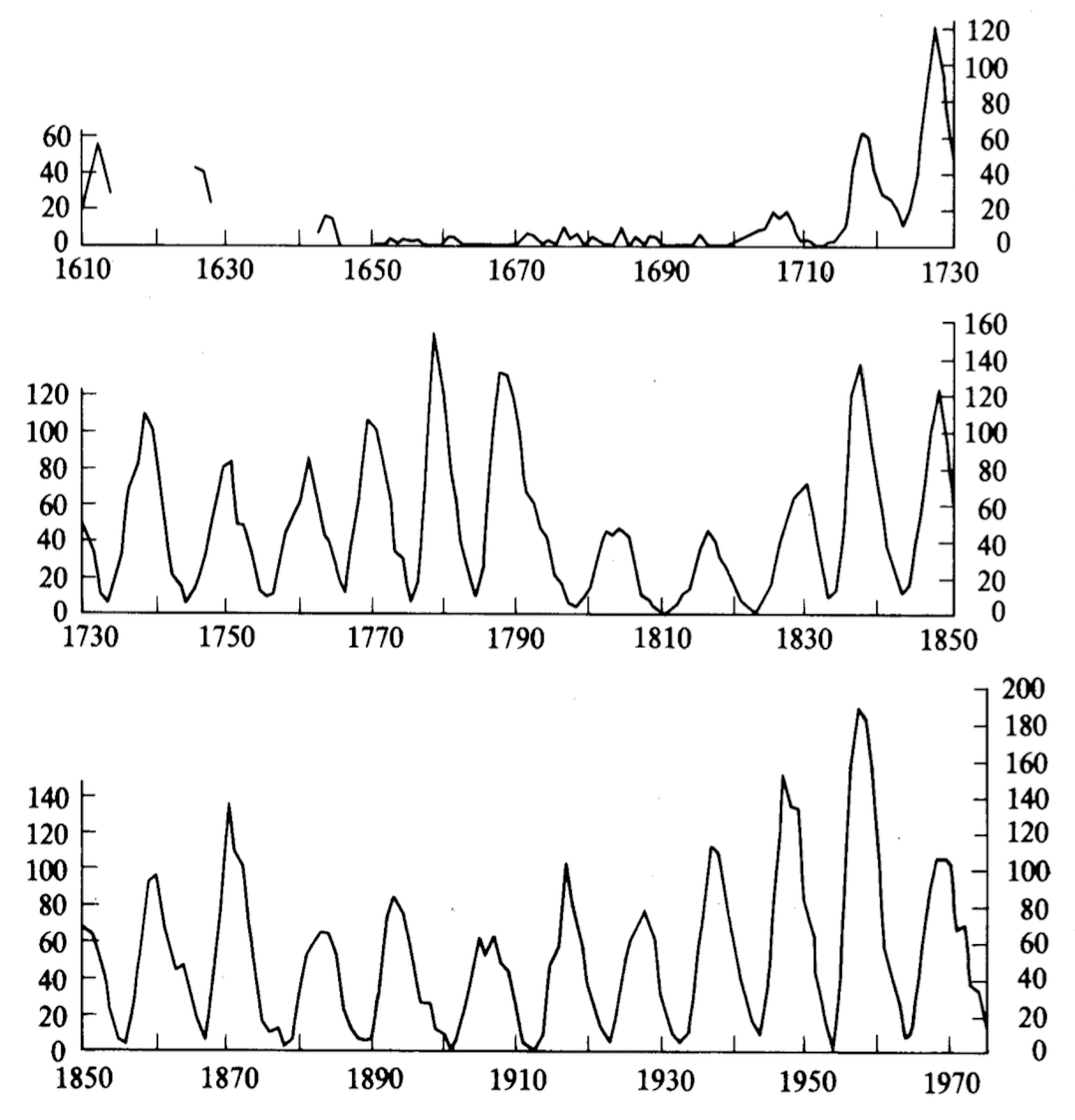}
    \caption{Annual mean sunspot number, 1610-1975. Figure from \cite{Moffatt78}.} \label{fig:moffatt}
\end{figure}

The process by which magnetic fields are produced in astrophysical objects, or more generally in conducting fluids, is the central subject of {{non-linear dynamo theory}}. A key subfield, known as {{kinematic dynamo theory}}, seeks to understand which classes of fluid flows can induce growth of a magnetic field when {{magnetic diffusivity}} is small.
Throughout the 20\textsuperscript{th} century, the problem of magnetic field growth has been extensively investigated in the context of the {{magnetohydrodynamics}} (MHD) equations, primarily from an applied and physical perspective, see~\cites{CG95,AK98,brandenburg2005astrophysical,Moffatt78,rudiger2006magnetic}. However, relatively few works have addressed this phenomenon rigorously from a mathematical standpoint, see~\cites{FV91,Gerardvaret05,GR07,Vishik86,NF_DV2025nonlinear}.

We study the kinematic dynamo equation (or passive vector equation) with a given velocity field  $u :  (0,\infty)\times \TT^3 \to \RR^3$
\begin{align} \label{passive-vector}
    \begin{cases}\tag{KDE}
    \partial_t B^\eps + u \cdot \nabla B^\eps - B^\eps \cdot \nabla u = \varepsilon \Delta B^\eps 
    \\
    \nabla \cdot B =0 \,.
    \end{cases}\qquad (t,x)\in (0,\infty)\times \TT^3 \,,
\end{align}
where  $\varepsilon >0$ is the magnetic diffusivity.
A central challenge in this field is to understand the mechanisms underlying fast dynamos—flows capable of sustaining magnetic field growth at rates that remain independent of magnetic diffusivity. Mathematically, this corresponds to an $\varepsilon$-independent exponential growth in time of the $L^2$ norm of $B^\varepsilon$, also referred to as the total magnetic energy. Interestingly enough, several mathematical definitions have been proposed in the literature (see, for instance, \cites{CG95, AK98, Arnold04, CZSV25, R25,NF_DV2025spectral}) for capturing this phenomenon, which, as will be shown in this paper, are \emph{not} equivalent when the velocity field $u$ is time-dependent. In particular, these definitions differ in how the limits $\varepsilon \to 0$ and $t \to \infty$ are taken. In what follows, we present a precise mathematical definition that captures the essential difficulty of the problem.
Let $u$ be  a bounded, Lipschitz continuous, divergence-free velocity field, $\eps \in (0,1)$ and  $B_{\text{in}}^\eps \in L^2 (\TT^3)$ the initial datum. Then, we define the \emph{dynamo growth rate} ${\gamma_\eps}= {\gamma_\eps} (B_{\text{in}}^\eps, \eps)$ of \eqref{passive-vector} as  
\begin{align} \label{d:gamma}  
{\gamma_\eps}  := \liminf_{t\to\infty}\frac1t \log \| B^\eps (t) \|_{L^2(\TT^3)}>0 \,. 
\end{align}
Conversely, we define the \emph{$\limsup$ dynamo growth rate} $ \overline{\gamma_\eps} = \overline{\gamma_\eps} (B_{\text{in}}^\eps, \eps)$ by 
\begin{equation} \label{d:gamma-over}
\overline{\gamma_\eps}:=\limsup_{t \to \infty} \frac{1}{t} \log\|B^\eps(t)\|_{L^2(\mathbb{T}^3)}>0,
\end{equation}
We say that $u$ is a fast dynamo (correspondingly $\limsup$ fast dynamo) if there exists  a sequence $\{ B_{\initial}^\eps \}_{\eps >0} \subset L^2$ such that the following holds true.

\[
\begin{aligned}
&\textbf{Fast dynamo:} \quad 
&\liminf_{\varepsilon \to 0} {\gamma_\varepsilon} > 0 \\[6pt]
&\textbf{lim\,sup Fast dynamo:} \quad
&\liminf_{\varepsilon \to 0} \overline{\gamma_\varepsilon} > 0
\end{aligned}
\]



The definition of a fast dynamo adopted here coincides with that given by Arnold in \cites{Arnold04, AK98}, whilst the definition of a $\limsup$ fast dynamo agrees with the fast dynamo definition introduced by Childress and Gilbert \cite[see (4.1.2) and (4.1.4)]{CG95}.
We observe that every fast dynamo is necessarily a $\limsup$ fast dynamo, since by definition ${\gamma_\varepsilon} \leq \overline{\gamma_\varepsilon}$; however, the converse does not always hold, unless the velocity field is autonomous or time-periodic, see Lemma \ref{lemma:equivalent}. In this work, we resolve the problem of finding a $\limsup$ fast dynamo, which corresponds to the formulation of the fast dynamo problem given in \cite{CG95}; however, we believe that establishing the existence of a genuine fast dynamo on $\TT^3$ is substantially more difficult.
 
\begin{theorem} \label{thm:main}
    There exists a divergence free velocity field $u \in L^\infty ((0,\infty) ; W^{1, \infty} (\TT^3))$ that is a  $\limsup$ fast dynamo on $\TT^3$.
\end{theorem}
As a corollary, since our velocity field $u$ satisfies $\gamma_\eps =0$ for all $\eps >0$, we also show that ${\gamma_\eps}<\overline{\gamma_\eps}$ is possible for time-dependent velocity fields.
\subsection{Literature and main challenges of fast dynamo action}

Chaotic velocity fields can be constructed using techniques from random dynamical systems \cites{BBPS22,BaxendaleRozovskii93,CZNF24}. 
In particular, the authors \cites{BaxendaleRozovskii93,CZNF24} prove that chaotic random dynamical systems can produce a universal \emph{ideal dynamo}---that is, a velocity field for which every nontrivial \(L^2\) initial datum yields an {exponentially growing} solution of \eqref{passive-vector} in the inviscid case \(\varepsilon = 0\). 
Such exponential growth is a \emph{necessary} condition for fast dynamo action under suitable assumptions as shown by Vishik \cite{Vishik89}, since a fast dynamo requires positivity of the ideal Lagrangian top Lyapunov exponent of \(u\). In fact, fast dynamo action requires the stronger condition of the velocity field having non-zero topological entropy, see \cite{Klapper_Young_1995}.
The central difficulty of the fast dynamo problem is that treating exponential growth for \(\varepsilon = 0\) perturbatively in \(\varepsilon\) is highly nontrivial.
In all known examples, ideal exponential growth is accompanied by the generation of increasingly fine spatial scales that are difficult to characterize along the evolution. 
These small scales can destroy large-scale coherence, potentially suppressing the net magnetic energy growth responsible for the dynamo effect. 
Understanding this interplay between the stretching term \(B^\varepsilon \!\cdot\! \nabla u\), which promotes growth, and the advective term \(u \!\cdot\! \nabla B^\varepsilon\), which transfers energy to smaller scales, lies at the heart of the fast dynamo problem.
From a spectral theory viewpoint, the main difficulty arises from the structure of the vector-transport operator. 
When \(u\) is autonomous, this operator has no discrete spectrum in $L^2$ and $C^0$ (see \cite{Chicone_Latushkin_Montgomery-Smith_1995}), so existing techniques from the singular perturbation theory for treating \(\varepsilon \Delta\) as a perturbation (such as those in \cites{K76,Albritton_Brué_Colombo_2022}) do not seem to be applicable.
In fact, it is known that exponential growth in the ideal case \(\varepsilon = 0\) does not in general imply fast dynamo action, as shown by any 2D velocity field with a hyperbolic point. A more subtle question, known as the \emph{flux conjecture} \cite{CG95}, asks whether exponential growth in the ideal case when averaged against a smooth test function (i.e.\ $|\langle B(t),\phi\rangle| \gtrsim \e^{\lambda t}$ for $\langle \cdot, \cdot \rangle$ the $L^2$ inner product, and $\phi$ some smooth test function) implies fast dynamo action. Whilst the validity of this conjecture is unclear in the case of fully time-dependent velocity fields $u$, in the case where $u$ is autonomous or time-periodic, it is supported by strong numerical evidence \cite{CG95}---however, there are no known rigorous proofs.
Kinematic dynamo theory is further constrained by structural obstructions to simplification.
Classical anti-dynamo theorems exclude fast dynamos under particular symmetries. 
Notably, Zeldovich’s theorem \cites{zeldovich1980magnetic,Zeldovich_1992} rules out fast dynamo action for flows with zero vertical component, while Cowling’s theorem \cite{Cowling33} forbids axisymmetric dynamos. 
These results underscore that a genuine fast dynamo must go beyond purely two-dimensional dynamics \cite{AK98}.
To date, the only rigorous proof of kinematic dynamo action on compact subdomains of $\mathbb{R}^3$ was provided in the recent work \cite{NF_DV2025spectral}. However, the growth rate ${\gamma_\eps}$ is of order $\eps^{1/3}$ in the diffusivity, a phenomenon known as \emph{slow dynamo action}.
Thus far, all rigorously established fast dynamo examples are defined on \(\RR^3\): see \cite{ZRMS84} for an unbounded velocity field, \cite{Gilbert88} for a discontinuous one, and the recent work \cite{CZSV25} for bounded, Lipschitz-continuous fields.
In \cite{CZSV25}, the authors show that the alpha-effect can produce exponential magnetic-field growth and, in particular, that suitably rescaled ABC flows (with \(\varepsilon\)-dependent scaling) yield exponential growth on \(\TT^3\). The authors use this result to construct a fast dynamo on $\RR^3$. 
The analysis of the alpha-effect for ABC flows plays a key role in the proof of Theorem~\ref{thm:main}. 
We also draw inspiration from the recent result of Rowan \cite{R25}, who constructs a smooth velocity field on $\TT^3$ that produces exponential magnetic growth along a diverging sequence of times and a vanishing subsequence of diffusivities. The improvement in Theorem~\ref{thm:main}, compared to the result in \cite{R25}, is that the growth we obtain is uniform as $\eps \to 0$ and does not occur only along a subsequence.
\vspace{-0.12cm}
\subsection{Strategy of the proof}
We consider, as a building block,  an ABC flow with $a=0$ and $b=c=1$ and with small amplitude $\delta >0$, i.e.
\[
U(\xx) = \delta \begin{pmatrix}
    \cos(y) \\
    \sin(x) \\
    \sin(y) + \cos(x)
\end{pmatrix}.
\]
We study spectral property of the kinematic dynamo operator associated to this velocity field in Section \ref{sec:ABC}, and we prove that the alpha-effect gives rise to solutions in the modal form 
$$B(t, \xx) = \exp (i \jj \cdot \xx + p t ) H(\xx) \,, \quad \jj \in \RR^3 \,, |\jj| \ll 1$$
that are exponentially growing in time with $\eps=1$. We also prove that any initial datum such that $P_\lambda  B \neq 0$, where $P $ is the Riesz projector onto the eigenvalue $\lambda $ with $\Re (\lambda)>0$, also gives rise to an exponentially growing solution. In particular, we prove that for $v= (-i, 1, 0)$ we have an exponentially growing solution with velocity field $U$ and initial datum  
$$ B_{\initial} = \exp (i \jj \cdot \xx) v \,.$$
This is inspired by the work \cite{CZSV25}. We further use perturbation theory and sectorial operator theory, see Section \ref{sec:sectorial} and Lemma \ref{lemma:sectorial2}, to deduce that solutions $B^\eps$ to \eqref{passive-vector} with $|\eps -1| \ll 1$ and initial datum $B (0) \in L^2$ such that 
$\| B_{\initial} - B (0) \|_{L^2} \ll 1$ satisfy
$$ \| B^\eps (t) \|_{L^2}  \gtrsim \e^{\frac {\Re (p )}{2} t} \| B_{\initial} \|_{L^2} \,.$$
Furthermore, we deduce this property also for the rescaling 
$$ B (t, \xx) \to B^{(n)} :=B(t, n \xx)\,, \quad U( \xx) \to U^{(n)}:= \frac{1}{n}U( n \xx) \,, \quad \eps \to \eps_n :=\frac{\eps}{n^2} \qquad \forall n \in \NN \,,$$
see Proposition \ref{prop:main} for a precise statement\footnote{Technically, we also rescale $U$ and $B$ with a fixed number $N_0 \in \NN$ sufficiently large so that they becomes periodic functions. Indeed, $B_{\initial}$ is not periodic because $|\jj|\ll 1$.}. Notice that the Lipschitz norm of $U$ with this rescaling is non-increasing.
We  construct the time-dependent velocity field $u$ for which growth of solutions will be guaranteed for $\eps \in (\eps_n)_n$, which is inspired by \cite{R25}. 
We fix a sequence that visits each $n \in \NN$ infinitely many times. More precisely, we fix 
a sequence $(a_k)_{k \ge 1}$ of natural numbers such that
\[
\forall n \in \mathbb{N}, \quad \#\{ k \in \mathbb{N} : a_k = n \} = \infty\,.
\] 
We denote by $(t_k)_k$ a suitable sequence of times to be chosen.
We crucially observe that the vector
$(0,0,1)$ is a stationary solutions to \eqref{passive-vector} with such velocity fields, which will be our initial condition in Theorem \ref{thm:main}. We notice that the average $(0,0,1)$ is preserved along the evolution of \eqref{passive-vector}. 
We then use the velocity field \( U_{g}^{(n)} \), defined in \eqref{d:velocity-generation}, which generates
\[
B^{(n)}_{\mathrm{in}}(\xx) = B_{\mathrm{in}}(n \xx)
\]
within time \(1\), starting from the spatially averaged initial state \((0,0,1)\).
With this newly generated initial condition, the solution to \eqref{passive-vector} exhibits exponential growth under the action of \( U^{(n)} \).
More precisely, if $k \in \NN$ is such that $a_{k+1} =n $, then we define the velocity field $u$ on $[t_{k}, t_{k+1}]$ as follows

\medskip

\medskip 

\begin{tikzpicture}[>=Stealth, x=1.18cm, y=1cm] 
\centering

  \def\tkmone{0}
  \def\Dtunits{12} 
  \pgfmathsetmacro{\tk}{\tkmone+\Dtunits}
  \pgfmathsetmacro{\tmid}{\tkmone+\Dtunits/2}
  \pgfmathsetmacro{\tplus}{\tmid+1.4}  

  \draw[->] (\tkmone-0.6,0) -- (\tk+0.6,0) node[below right] {$t$};

  \fill[gray!20,  draw=black] (\tkmone,0.6) rectangle (\tmid,1.4);
  \node at ({(\tkmone+\tmid)/2},1.0) {$u \equiv 0$};

  \fill[blue!15,  draw=black] (\tmid,0.6)   rectangle (\tplus,1.4);
  \node at ({(\tmid+\tplus)/2},1.0) {$u\equiv U_{g}^{(n)}$};

  \fill[green!20, draw=black] (\tplus,0.6)  rectangle (\tk,1.4);
  \node at ({(\tplus+\tk)/2},1.0) {$u \equiv U^{(n)}$};

  \draw (\tkmone,0.12) -- (\tkmone,-0.12);
  \draw (\tmid,0.12)   -- (\tmid,-0.12);
  \draw (\tplus,0.12)  -- (\tplus,-0.12);
  \draw (\tk,0.12)     -- (\tk,-0.12);

  \node[below=4pt] at (\tkmone,0) {\scriptsize $t_{k}$};

  \node[below=6pt, xshift=-4pt] at (\tmid,0)
    {\scriptsize $t_{k}+\tfrac{\Dt}{2}$};

  \node[below=6pt, xshift=+4pt] at (\tplus,0)
    {\scriptsize $t_{k}+\tfrac{\Dt}{2}+1$};

  \node[below=4pt] at (\tk,0) {\scriptsize $t_{k+1}$};

\end{tikzpicture}

Assuming again that $a_{k+1} = n$, then if we choose the sequence $(t_k)_k$ as in  \eqref{eq:tk}, we can prove that the evolution of the magnetic energy of $B^{\eps_n}$, solution to \eqref{passive-vector}, satisfies the following

\medskip
\medskip

   \begin{tikzpicture}[>=Stealth, x=1.2cm, y=1.2cm]

  \def\tk{0}              
  \def\Dtunits{10}        
  \pgfmathsetmacro{\tkone}{\tk+\Dtunits}      
  \pgfmathsetmacro{\tmid}{\tk+\Dtunits/2}     
  \pgfmathsetmacro{\tflatend}{\tmid+1}        

  \draw[->] (\tk-0.6,0) -- (\tkone+0.6,0) node[below right] {$t$};
  \draw[->] (\tk-0.6,0) -- (\tk-0.6,2.5) node[above] {$\|B^{\eps_n}(t)\|_{L^2}$};

  \draw[densely dashed] (\tmid,0) -- (\tmid,2.4);
  \draw[densely dashed] (\tflatend,0) -- (\tflatend,2.4);

  \draw (\tk,0.12)       -- (\tk,-0.12);
  \draw (\tmid,0.12)     -- (\tmid,-0.12);
  \draw (\tflatend,0.12) -- (\tflatend,-0.12);
  \draw (\tkone,0.12)    -- (\tkone,-0.12);

  \node[below=4pt] at (\tk,0)        {\scriptsize $t_k$};
  \node[below=6pt, xshift=-2pt] at (\tmid,0)
    {\scriptsize $t_k+\tfrac{\Dt}{2}$};
  \node[below=6pt, xshift=+2pt] at (\tflatend,0)
    {\scriptsize $t_k+\tfrac{\Dt}{2}+1$};
  \node[below=4pt] at (\tkone,0)     {\scriptsize $t_{k+1}$};

  \draw[thick]
    plot[domain=0:1, samples=60]
      ({\tk + (\tmid-\tk)*\x},
       {0.4 + (2.2-0.4)*(1-\x)^2});

  \draw[thick]
    (\tmid,0.4) -- (\tflatend,0.45);

  \draw[thick]
    plot[domain=0:1, samples=60]
      ({\tflatend + (\tkone-\tflatend)*\x},
       {0.45 + (2.3-0.45)*\x^2});

  \node at ({(\tk+\tmid)/2},2.1)         {\scriptsize decay};
  \node at ({(\tmid+\tflatend)/2},0.9)   {\scriptsize almost constant};
  \node at ({(\tflatend+\tkone)/2},2.1)  {\scriptsize growth};

\end{tikzpicture}


In particular, we can prove that 
$$ \| B^{\eps_n} (t_{k+1}) \|_{L^2}  \gtrsim \e^{\frac {\Re (p)}{3} t_{k+1}} \| B_{\initial} \|_{L^2} \,.$$
Finally, with  Proposition \ref{prop:main} in hand we can deduce a similar evolution  for the solution $B^\eps$ with $\eps \in [\eps_n, \eps_{n-1})$. Indeed, Proposition \ref{prop:main} encodes the continuity of the growth with respect to small perturbation of initial datum and of $\Delta$ in the equation \eqref{passive-vector} and deduce that 
$$  \| B^{\eps} (t_{k+1}) \|_{L^2}  \gtrsim \e^{\frac {\Re (p)}{4} t_{k+1}} \| B_{\initial} \|_{L^2} \,, \qquad \forall \eps \in [\eps_n , \eps_{n-1}) \,. $$

\section{Preliminaries on spectral theory and sectorial operators} \label{sec:sectorial}

 We present some background in the theory of sectorial operators that is well known, see \cite{K76}. For the convenience of the reader we present the proofs of the lemmas in Appendix \ref{sec:appendix}. Firstly, for $(X, \| \cdot \|)$ a Banach space and $T:\mathcal{D}(T) \subset X \to X$ a closed operator with dense domain we define the Riesz projector associated to $\gamma$ a contour  that is contained in the resolvent $\gamma \subset \rho(T)$ as 
\begin{equation}
\label{eq:P abstract}
P=\frac{1}{2 \pi i} \int_{\gamma} (z-T)^{-1} \dd z \,.
\end{equation}
If in the interior of the contour $\gamma$ there is a unique eigenvalue $\lambda$, then we may say that $P$ is the Riesz projector corresponding to the eigenvalue $\lambda$. We recall the definition of sectorial operators, see Figure \ref{fig:sectorial}.
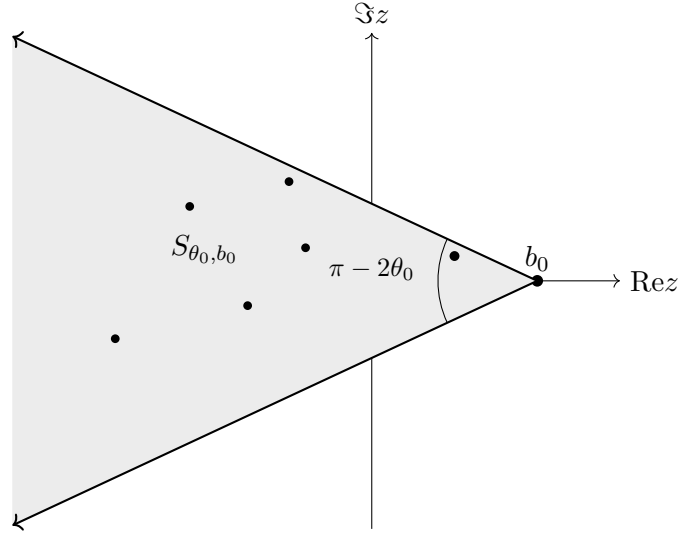
\begin{figure}[h]
    \centering
    \begin{tikzpicture}[scale=1.1]

        \draw[->] (-4,0) -- (3,0) node[right] {$\Re z$};
        \draw[->] (0,-3) -- (0,3) node[above] {$\Im z$};

        \coordinate (b0) at (2,0);
        \fill (b0) circle (2pt) node[above] {$b_0$};

        \def\alpha{25}

        \def\R{7}

        \coordinate (U) at ($(b0)+(180-\alpha:\R)$);
        \coordinate (L) at ($(b0)+(180+\alpha:\R)$);

        \path[fill=gray!15] (b0) -- (U) -- (L) -- cycle;

        \draw[thick,->] (b0) -- ($(b0)+(180-\alpha:7)$);
        \draw[thick,->] (b0) -- ($(b0)+(180+\alpha:7)$);

        \node[below left, yshift=-1.05cm] at ($(b0)!0.55!(U)$) {$S_{\theta_0,b_0}$};

        \fill (-0.8,0.4)  circle (1.5pt);
        \fill (-1.5,-0.3) circle (1.5pt);
        \fill (-2.2,0.9)  circle (1.5pt);
        \fill (-3.1,-0.7) circle (1.5pt);
        \fill (-1.0,1.2)  circle (1.5pt);

        \fill (1.0,0.3) circle (1.7pt);


        \draw[thin] (b0) ++(180-\alpha:1.2) arc (180-\alpha:180+\alpha:1.2);
        \node at ($(b0)+(-2.0,0.15)$) {\small $\pi - 2\theta_0$};

    \end{tikzpicture}
    \caption{The spectrum $\sigma(T)$ of a sectorial operator contained in the sector $S_{\theta_0,b_0}$.} \label{fig:sectorial}
\end{figure}
 
\begin{definition} \label{d:sectorial}
Let $(X, \|\cdot \|)$ be a separable Banach space and let $T:\mathcal{D}(T) \subset X \to X$ be a closed operator with dense domain. We say that $T$ is Sectorial of angle $\theta_0 \in (0, \frac{\pi}{2})$ if there exists a constant $b_0>0$ so that the spectrum of $T$ is contained in the sector $S_{\theta_0,b_0}=\{z \in \mathbb{C}:\mathrm{Arg}(b_0-z) \in [-\frac{\pi}{2}+\theta_0,\frac{\pi}{2}-\theta_0]\}$, and there exists a constant $C>0$ so that 
\begin{equation}
\|(\lambda-T)^{-1}\|\leq C \dist(\lambda,S_{\theta_0,b_0})^{-1}
\end{equation}
for all $\lambda \notin S_{\theta_0,b_0}$.
\end{definition}

A key property of sectorial operators is that their generated semigroups admit an explicit representation formula via the holomorphic functional calculus for sectorial operators. 
\begin{theorem}
Let $T$ be sectorial of angle $\theta_0$ with constant $b_0$. For any $\theta <\theta_0$, $b>b_0$, denote by $\gamma_{\theta,b}$ the contour obtained by traversing the boundary of the sector $S_{\theta,b}$ counterclockwise. Then, the semigroup generated by $T$ is given by 
\begin{equation*}
S(t)=\frac{1}{2 \pi i}\int_{\gamma_{\theta,b}}(z-T)^{-1} \e^{zt} \dd z.
\end{equation*}
\end{theorem}
As a result of this representation formula, we deduce the following result on the semigroup generated by sectorial operators with compact resolvents.
\begin{lemma}
\label{lemma:sectorial1}
Let $(T,\mathcal{D}(T))$ be a sectorial operator with compact resolvent on some Hilbert space $H$, with $\sigma(T) \subset S_{\theta_0,b_0}$ as in Definition \ref{d:sectorial}. 
Then, for all $p < b_0$, and all constants $\zeta>0$ small enough, there exist contours $\gamma_1, \gamma_2$ so that $\gamma_1 \subset \{\Re(z)\leq p-\zeta\}$, $\gamma_2 \subset \{\Re(z) \geq p-\frac{\zeta}{2}\}$ and the Semigroup generated by $T$ is equal to 
\begin{equation}
\e^{T t}=\frac{1}{2 \pi i}\int_{\gamma_1} (z-T)^{-1} \e^{zt} \dd z+\frac{1}{2 \pi i}\int_{\gamma_2} (z-T)^{-1} \e^{zt} \dd z=:S_1(t)+S_2(t) \,.
\end{equation}
Furthermore, there holds $\liminf_{t \to \infty}\frac{1}{t}\|\e^{T t}x\| \geq p-\frac{\zeta}{2}$ if and only if $Px \neq 0$, where $P$ is the Riesz projector associated to $\gamma_2$ as in \eqref{eq:P abstract}. 
\end{lemma}

As a consequence of the previous lemma, we obtain a continuity property for sectorial operators with compact resolvent of a particular structure. 
Specifically, we will apply the lemma below with 
\begin{align} \label{eq:applications}
    T_1 = \nabla \times (U \times \cdot), 
\qquad 
T_0 = \varepsilon \Delta, 
\qquad 
\varepsilon > 0\,,
\end{align}
to deduce that the operator 
\[
\kappa T_0 + T_1
\]
exhibits exponential growth for $\kappa \approx 1$, assuming the exponential growth for $\kappa=1$. 
Moreover, this growth is uniform with respect to small perturbations of the initial data. This lemma is crucial in order to obtain the result $\liminf_{\eps \to 0}$ in Theorem \ref{thm:main}.

\begin{lemma}
\label{lemma:sectorial2}
Let $T(\kappa)=\kappa T_0 +T_1$, where $(T_0, \mathcal{D}(T_0))$ is a negative, self-adjoint operator with compact resolvent, and for any $\delta>0$ there exists a constant $C(\delta)$, continuous in $\delta$ so that for all $x \in \mathcal{D}(T_0)$ it holds 
\begin{equation}
\|T_1 x\|\leq \delta \|T_0 x\|+C(\delta)\|x\|.
\end{equation}
Suppose further that there exists $x_0 \in H$ so that $\liminf_{t \to \infty}\frac{1}{t}\|e^{T(1) t} x_0\| = p >0$. Then, there exist constants $\kappa_0>0$, $\eta_0>0$, $C>0$, $T>0$ depending only on $\frac{x_0}{\|x_0\|}, T(1)$, so that for all $|1-\kappa|\leq \kappa_0$, $\|x-x_0\|\leq \eta_0 \| x_0\|$ it holds
\begin{equation}
\|\e^{T(\kappa)t}x\|\geq C \e^{\frac{p}{2}t}\|x\|, \quad \forall t \geq T\,.
\end{equation}
\end{lemma}

We develop a Taylor expansion of the eigenvalues with respect to $j$ of the operator $T_0 + j T_1$ as $j \to 0$.

\begin{lemma} \label{lemma:kato-convergence}
Let $T(j)=T_0+j T_1$, where $T_1$ is relatively bounded with respect to the closed operator $(T_0,\mathcal{D}(T_0))$ on some Banach space $X$. Assume that $p$ is an isolated, semisimple eigenvalue of $T_0$, and assume further that the eigenvalues of the operator $PT_1P:PX \to PX$ are all distinct, where $P$ is the Riesz projector corresponding to the eigenvalue $p$ of $T_0$. Then, $T(j)$ admits eigenvalues of the form 
\begin{equation}
p_i(j)=p+j\mu_i +o(j),
\end{equation}
where $\mu_i$ are the eigenvalues of the operator $PT_1P:PX \to PX$. Let $P_i^{(1)}$ be the Riesz projectors associated to the eigenvalues of $PT_1P$, and denote by $P_i(j)$ the projectors associated to the eigenvalues $p_i(j)$, which act on $X$ by $P_i^{(1)} x=P_i^{(1)} Px$. Then, in the strong operator topology
$$P_i(j)  \to P_i^{(1)}\,, \quad \text{ as } \,  j \to 0\,.$$
\end{lemma}

Finally, we show that for autonomous or time-periodic vector fields on $\mathbb{T}^3$ (or more generally on any compact domain), the dynamo growth rate $\gamma_\eps$ in \eqref{d:gamma} coincides with the $\limsup$ dynamo growth rate $\overline{\gamma_\eps}$ in \eqref{d:gamma-over}. 
\begin{lemma}
\label{lemma:equivalent}
Let $u \in L^\infty([0,\infty),W^{1,\infty}(\mathbb{T}^3))$ be a time periodic divergence--free velocity field. Then, the following holds true 
$${\gamma_\eps}=\overline{\gamma_\eps}\,, \qquad \forall \eps >0 \,.$$
\end{lemma}
\begin{proof}
For the proof we set $\eps=1$ and $u$ be a $1$ periodic velocity field to ease readability. 
Let $S(t,s)$ be the propagator associated to the 
\begin{equation*}
\partial_t B =\nabla \times (u\times B)+\Delta B,
\end{equation*}
in other words, $S(t,s)B_0$ is the solution to \eqref{passive-vector} with velocity field $u$ and initial condition $B_0$ starting from time $s$ and posing the equation for $t > s$.
Furthermore, let $T=S(1,0)$ be the monodromy operator associated to the problem. 
We claim that 
$$\log r(T) \leq {\gamma}_1 \leq \overline{\gamma_1} \leq \log r(T)\,,$$ where $r(T)$ denotes the spectral radius of $T$. We firstly prove the lower bound and without loss of generality we assume that $r(T)>0$. Since $T$ is compact, we can pick an eigenfunction $B$ with eigenvalue $|\lambda|=r(T)$, and then set $B(t)=S(t,0)B$. Clearly $\|B(n)\|_{L^2}=|\lambda|^n \|B\|_{L^2}$, so it remains to obtain a lower bound in the intermediate region $t \in (n,n+1)$. By 
standard well-posedness theory, the map $t \mapsto S(t,0)B$ is continuous from $[0,1] \mapsto L^2(\mathbb{T}^3)$ and never vanishing. Thus, $\|B(t)\|_{L^2}$ attains a strictly positive minimum on $[0,1]$, which we call $m>0$. Thus, for $t \in [n,n+1]$ we bound 
$$
\|B(t)\|_{L^2}=\|S(t,n)\lambda^n B\|_{L^2}=|\lambda|^n \|S(t-n,0)B\|_{L^2} \geq m|\lambda|^n.
$$
Therefore, it follows that $\gamma_1 \geq \log (|\lambda|) = \log (r(T)) $.
The upper bound is a straightforward application of the  Gelfand spectral radius formula (see e.g. \cite{K76}) 
\begin{equation*}
r(T)=\limsup_{n \to \infty}\|T^n\|^{1/n} 
\end{equation*}
and the upper bound $\|B(t)\|_{L^2} \leq  \|B(n)\|\exp({\|u\|_{L^\infty_t W^{1,\infty}_x}})$ for any $t \in [n, n+1]$ by energy estimates.
Noting that this upper bound also holds in the case 
$r(T)=0$, we conclude the proof.
\end{proof}

\section{Preliminaries on rescaled ABC flows} \label{sec:ABC}
In this section, we recall some spectral properties of ABC flows that have been studied in \cite{CZSV25} and prove a few more properties we need for proving Theorem \ref{thm:main} about ABC flows with $a=0$ and $b=c=1$. We look for solutions in the modal form
$$B(t, \xx) = \exp (i \jj \cdot \xx + pt ) H(\xx) $$
of the kinematic dynamo equation \eqref{passive-vector} with $\eps =1$. The goal is to prove existence of a modal form solution with $\Re (p) >0$. We observe that $B$ is a solution to \eqref{passive-vector} with $\eps =1$ if and only if $H$ solves the eigenvalue problem
$$\cL(\jj) H = p H \qquad \text{and} \qquad (\nabla + i \jj) \cdot H =0 \,, $$
where
$$\cL (\jj) = (\nabla +i\jj)\times (  u\times H)+(\nabla +i\jj)^2H \,.$$
For convenience, we expand $\cL (\jj)$ in $\jj$ as 
$$ \cL(\jj) = \cL_0 + |\jj| \cL_1  - |\jj|^2 \Id $$
where 
\begin{align}
    \cL_0 & = \nabla \times ( u \times H) + \Delta H
    \\
    \cL_1 & = i \frac{\jj}{|\jj|} \times ( u \times H) + 2 i \frac{\jj}{|\jj|} \cdot \nabla H \,.
\end{align}

We recall that for any $u$ such that  $\| u \|_{W^{1,\infty}} < 1$ we have that $\Ker (\cL_0)$ is three dimensional, see \cite[Lemma 2.1]{CZSV25}. We also recall the following result from  \cite[Lemma 2.2]{CZSV25}.

\begin{lemma} \label{lemma:expansion-eigen}
There exist $\delta, \delta_1 >0$ small enough such that the following holds true. For any $|\jj| < \delta_1$ and $u \in W^{1, \infty} (\TT^3)$ with zero average and $\| u \|_{W^{1, \infty}}< \delta$, $\cL (\jj)$ admits eigenvalues of the form
\begin{equation}
p_\ell(\jj)=|\jj|\mu_\ell+o(|\jj|) \in \CC \,, \qquad \ell=1,2,3 \,,
\end{equation}
where $\mu_\ell$ are the (possibly repeated) eigenvalues of the operator 
\begin{equation}
P\mathcal{L}_1 P:PL^2(\mathbb{T}^3) \to PL^2(\mathbb{T}^3)\,,
\end{equation}
where $P$ is the Riesz projector, associated to the eigenvalue $0$, meaning that it projects onto $\Ker (\cL_0)$ which is a three dimensional space and $P L^2 (\TT^3) = Ker (\cL_0) \cong \CC^3$.
\end{lemma}

From now on for convenience of the notation we assume that $\delta \in (0,1)$ and
$$ u = \delta U \,, \quad \text{with } \quad  \| U \|_{W^\infty} \leq 1\,.$$
Under this assumption on $U$ we have that the operator $\Delta + \nabla \times (\delta U \times \cdot ) = \Delta (\Id + \Delta^{-1} \nabla \times (\delta U \times \cdot ))$ is invertible. 
For any $H \in L^2 (\TT^3)$ we  denote the average of the vector field as
$$ \langle H \rangle = \fint_{\TT^3} H $$
and we denote by $L^2_0 (\TT^3)$ the space of $L^2$ vector field with zero average. Then,  we define
$$\cS : \CC^3 \to L^2_0 (\TT^3)$$
so that for any $\cS(v) $ is the unique zero average solution to 
\begin{equation} 
    \nabla \times (\delta U \times \cS(v)) + \Delta \cS(v) = \nabla \times (   v \times \delta U)\,.
\end{equation}
By Neumann series expansion using that $\Delta + \nabla \times (\delta U \times \cdot ) = \Delta (\Id + \delta \Delta^{-1} \nabla \times ( U \times \cdot ))$ we have that $S(v)$ satisfies
\begin{equation} \label{eq:expansion-S}
    \| \cS(v) - \delta \Delta^{-1} (\nabla \times  (v \times U) ) \|_{L^2} \leq C \delta^2 |v| \,, \qquad \| \cS (v) \|_{L^2} \leq C \delta |v| \,.
\end{equation}
with a constant $C>0$ depending only on $\|U \|_{W^\infty}\leq 1$.
Therefore, from this and \cite[Lemma 2.3]{CZSV25} we deduce the following result.
\begin{lemma}  \label{lemma:projector}
Let $U \in W^{1, \infty} (\TT^3)$ with $\| U \|_{W^{1, \infty}} \leq 1$ and $\delta \in (0,1)$. Then, 
The Riesz projector $P$ acts on $L^2(\mathbb{T}^3)$ via 
\begin{equation}
PH=\langle H \rangle +\cS(\langle H\rangle),
\end{equation}
In particular, the Riesz projector $P$ admits the asymptotic expansion for any $\delta \in (0,1)$, meaning that there exists a universal constant $C>0$ such that  
$$ \| P H - \langle H \rangle -\delta\Delta^{-1} \nabla \times (\langle H\rangle \times U  )\|_{L^2} \leq C \delta^2 |\langle H \rangle |  \,. $$
\end{lemma}

We now fix the ABC flow with $a=0$ and $b=c=1$ that is independent on $z$. More precisely, we define
\begin{align}
    U(x,y)=\begin{pmatrix} \label{d:ABC-fixed}
\cos(y)\\
\sin(x)\\
\sin(y)+\cos(x) \end{pmatrix}\,. 
\end{align}
For such velocity field it can be computed explicitly the matrix $M_0 \in \CC^{3 \times 3}$ defined by 
$$ M_0(v) = i \frac{\jj}{|\jj|} \times  \fint ( \delta U \times (\delta \Delta^{-1} \nabla \times (v \times U)) ) \,, \qquad \forall v \in \CC^3\,, $$
and fixing $\jj = (0,0,j_3)$ and using Fourier series (thanks to the fact that $U$ is supported on modes $|k|=1$), we can compute explicitly 
$$M_0 =i \delta^2 
\begin{pmatrix} 
0&-  1&  0\\ 
1 & 0& 0   \\ 
0 &0 & 0 
\end{pmatrix}
$$
whose eigenvalues are  $\pm \delta^2$ and $0$.
We now compute the eigenvalues $\mu_\ell$ given in Lemma \ref{lemma:expansion-eigen} of $P \cL_1 P$, where $P$ is the Riesz projector onto $\Ker(\cL_0)$.
By Lemma \ref{lemma:projector} we have 
$P\cL_1 P H = \langle \cL_1 P H \rangle + \cS (\langle \cL_1 P H \rangle)  $.
Using also that $U$ is zero average we deduce that $\langle \cL_1 P H\rangle=\frac{i \jj}{|\jj|}\times \langle \delta U \times \cS(\langle H \rangle)\rangle  $
and by the expansion of $\cS (v)$ given by \eqref{eq:expansion-S} we conclude  
\begin{align} \label{eq:bound-M-0}
    \|M_0 (\langle H \rangle )  - P \cL_1 P H \|_{L^2} \leq  \|M_0 (\langle H \rangle )  - \langle P \cL_1 P H \rangle \|_{L^2} + \| \cS (\langle P \cL_1 P H \rangle)\|_{L^2}   \leq C \delta^3 \,.
\end{align}
Therefore, since the operator $M_0$ has a positive eigenvalue $\delta^2$ we also deduce that $P\cL_1 P$  has an eigenvalue $\mu_1$ such that $\Re (\mu_1) > \delta^2/2 $ for $\delta >0 $ sufficiently small.



\begin{lemma} \label{lemma:riesz-P-1}
Let $P^{(1)}$ be the Riesz projector of $P \cL_1 P$ onto the eigenvalue $\mu_1$ such that $\Re (\mu_1) > \delta^2/2$. Then, for any $H \in L^2 (\TT^3) $, there exists $\mathcal{O} (\delta) H \in L^2 (\TT^3)$ with $\|\mathcal{O}(\delta)H\|_{L^2(\mathbb{T}^3)} \leq C \delta |\langle H \rangle|$ such that the following holds true 
\begin{equation}
P^{(1)} H=\blue{-}\begin{pmatrix}
\frac{1}{2} & -\frac{i}{2} & 0\\
\frac{i}{2} & \frac{1}{2} & 0 \\
0 & 0 & 0
\end{pmatrix}\langle H\rangle +\mathcal{O}(\delta)H \,.
\end{equation}
\end{lemma}
\begin{proof}
    We consider $\gamma : [0,1] \to \CC$ to be a contour of the eigenvalue $\mu_1$, which contains only the eigenvalue $\mu_1$ and such that  $\Gamma = \{ \lambda \in \CC : \exists t \in [0,1]: \gamma(t) =\lambda\} \subset \rho (P \cL_1 P)$. To do so one can just consider a circle of radius $\frac{\delta^2}{2}$ around the eigenvalue the point $\delta^2$, since $\mu_1$ is in the discrete spectrum.
We now use the short hand notation for the matrix $M \in \CC^{3 \times 3}$ defined as
$$M (v) = \frac{i \jj}{|\jj|}\times \langle \delta U \times \cS( v )\rangle \,, \qquad \forall v \in \CC^3 \,. $$
A direct computation shows that for any $H \in PL^2(\TT^3)$, i.e. by Lemma \ref{lemma:projector} $H= \langle H \rangle + \cS (\langle H\rangle)$, and any $\lambda \in \Gamma$ we have $(P \cL_1 P - \lambda)^{-1} H = (M-\lambda)^{-1}\langle H\rangle+\cS((M-\lambda)^{-1}\langle H\rangle)$. Indeed, by observing that 
$$P\cL_1 P F = \langle \cL_1 PF\rangle + \cS (\langle \cL_1 PF\rangle) = M (\langle F \rangle) + \cS (M (\langle F \rangle)) \,, \qquad \forall F \in L^2$$
if we set $F= (M-\lambda)^{-1}\langle H\rangle+S((M-\lambda)^{-1}\langle H\rangle)$,
we deduce
\begin{align}
(P\cL_1P-\lambda) F & =M\langle F \rangle +\cS(M\langle F\rangle)-\lambda \langle F \rangle -\lambda \cS(\langle F \rangle )\\
& =(M-\lambda)\langle F \rangle +\cS((M-\lambda)\langle F\rangle)=\langle H\rangle +\cS(\langle H\rangle)\,.
\end{align}

 Then, we have 
\begin{equation}
\frac{1}{2 \pi i}\int_{\Gamma} (P\cL_1P-\lambda)^{-1} H \dd \gamma=\frac{1}{2 \pi i}\int_{\Gamma}(M-\lambda)^{-1}\langle H \rangle \dd \gamma +\cS\left (\frac{1}{2 \pi i}\int_{\Gamma}(M-\lambda)^{-1}\langle H \rangle \dd \gamma \right)
\end{equation}
and using \eqref{eq:expansion-S} we just have to consider 
$$\frac{1}{2 \pi i}\int_{\Gamma}(M-\lambda)^{-1}\langle H \rangle \dd \lambda$$
up to an error of size $\mathcal{O} (\delta) H$.
Thus we compute
\begin{align}
&(M-\lambda)^{-1}= (M_0 - \lambda + (M-M_0))^{-1}=  (\Id + (M_0 - \lambda)^{-1}(M - M_0) )^{-1} (M_0 - \lambda)^{-1}
\end{align}
and since $\|  (M_0 - \lambda)^{-1}(M - M_0)  \|_{L^2 \to L^2 } \leq C \delta^{-2} \delta^3 = C \delta$ we reduce the computation to 
$$ \frac{1}{2 \pi i}\int_{\Gamma}(M_0-\lambda)^{-1}\langle H \rangle \dd \lambda \,,$$
up to errors of size $\mathcal{O} (\delta) H$.
Parametrizing the contour as $\gamma = \delta^2+\frac{1}{2}\delta^2 \eta$, where $\eta$ traces out a circle of radius $1$ in the complex plane denoted by $\Gamma_1$ by changing variable we deduce
\begin{equation}
\frac{1}{2 \pi i}\int_{\Gamma}(M_0-\lambda)^{-1} \dd \gamma = \frac{1}{2\pi i }\int_{\Gamma_1} (M_0-(1+\frac{1}{2}\eta))^{-1} \dd \eta \,.
\end{equation}

By an explicit computation $(M_0-(1+\frac{1}{2}{\eta}))^{-1}$ is given by 
\begin{equation}
(M_0-(1+\frac{1}{2}\eta))^{-1}=\begin{pmatrix}
\frac{1+\frac{1}{2}\eta}{1-(1+\frac{1}{2}\eta)^2} & \frac{i}{(1+\frac{1}{2}\eta)^2-1} & 0 \\
\frac{-i}{(1+\frac{1}{2}\eta)^2-1} & \frac{1+\frac{1}{2}\eta}{1-(1+\frac{1}{2}\eta)^2}  & 0\\
0 & 0 & -\frac{1}{(1+\frac{1}{2}\eta)}
\end{pmatrix},
\end{equation}
so integrating this around a circle $\Gamma_1$ and using Cauchy integral formula yields 
\begin{equation} 
\frac{1}{2\pi i }\int_\Gamma (M_0-(1+\frac{1}{2}\eta))^{-1} \dd \eta  =
\begin{pmatrix}
\frac{1}{2} & -\frac{i}{2} & 0\\
\frac{i}{2} & \frac{1}{2} & 0 \\
0 & 0 & 0
\end{pmatrix} 
\end{equation}
concluding the proof.
\end{proof}

\begin{lemma}
\label{lemma:initial growth}
Let $B_{\initial} (x,y,z) = \e^{i j z} (-i,1,0)$, then there exist $\delta, j \in (0,1)$ sufficiently small such that the following holds true. There exists $\lambda >0$ such that the solution $B$ to \eqref{passive-vector}, with $\eps =1$ and velocity field $u = \delta U$ with $U$ the ABC flow with $a=0$ and $b=c=1$ defined in \eqref{d:ABC-fixed}, satisfies
$$ \| B(t) \|_{L^2} \geq \e^{\lambda t} \| B_{\initial} \|_{L^2} \qquad \forall t \geq 0 \,.$$
\end{lemma}

\begin{proof}
Let $P$ be the Riesz projector of $\cL_0$ onto $\Ker(\cL_0)$ and $P^{(1)}$ be the Riesz projector of $P \cL_1 P$ onto the eigenvalue $\mu_1$ with $\Re (\mu_1) > \delta^2/2$ thanks to Lemma \ref{lemma:riesz-P-1} of the form 
\begin{equation}
P^{(1)} H=\blue{-}\begin{pmatrix}
\frac{1}{2} & -\frac{i}{2} & 0\\
\frac{i}{2} & \frac{1}{2} & 0 \\
0 & 0 & 0
\end{pmatrix}\langle H\rangle +\mathcal{O}(\delta)H \,, \qquad \forall H \in L^2 \,.
\end{equation}
We define $H_{\initial} = (-i,1,0)$ and we notice that $P H_{\initial} \neq 0$, since $H_{\initial} = \langle H_{\initial} \rangle \neq 0$. Then, $P^{(1)} H_{\initial} = P^{(1)} P H_{\initial}$ and 
$$ P^{(1)} P H_{\initial}=\blue{-}\begin{pmatrix}
\frac{1}{2} & -\frac{i}{2} & 0\\
\frac{i}{2} & \frac{1}{2} & 0 \\
0 & 0 & 0
\end{pmatrix}\begin{pmatrix}
-i\\
1\\
0
\end{pmatrix}  +\mathcal{O}(\delta)  = -\begin{pmatrix}
-i\\
1\\
0
\end{pmatrix} + \mathcal{O}(\delta)\,. $$
Therefore, for $\delta >0$ small enough, $P^{(1)} H_{\initial} \neq 0$. Hence, letting $P_{\jj}$ be the Riesz projection onto the eigenvalue $|\jj|(\delta^2+\mathcal{O}(\delta^3))$ of $\mathcal{L}(\jj)$\footnote{Here $\jj = (0,0,j)$.}, we deduce that for $|\jj|$ small enough it holds $P_{\jj} H_{\initial} \neq 0$ thanks to Lemma \ref{lemma:kato-convergence}. Finally, we note that $\mathcal{L}(j)$ generates an analytic semigroup, and has compact resolvent. Therefore, by Lemma \ref{lemma:sectorial1} we immediately deduce the claimed growth for $|\jj|$, $\delta$ small enough.
\end{proof}

\section{Main proposition}
In this section, we state and prove the main proposition concerning our building block velocity field, which asserts that the exponential growth of solutions persists under small perturbations of $\eps$ and of the initial data.
More precisely, we fix  $N_0 \in \NN$ and $\delta >0$ given in Proposition \ref{prop:main} and
for any $n \in \NN$ let $U^{(n)}  \in C^\infty (\TT^3)$ be the divergence-free velocity field
\begin{equation} \label{d:building-block}
U^{(n)} (\xx) =  \frac{\delta}{N_0 n} \begin{pmatrix}
\cos(n N_0 y)\\
\sin(n N_0 x)\\
\sin(n N_0 y)+\cos(n N_0 x)
\end{pmatrix}
\end{equation} 
which is a rescaling of a  particular  choice of an ABC flow.
We also denote by $\cS^\eps_n (t)$ the semigroup generated by  equation \eqref{passive-vector} with velocity field $U^{(n)}$. In particular, we denote by 
$$\cS^\eps_n (t) B_{\initial} : = B^\eps (t)$$ the solution to \eqref{passive-vector} under the action of the  velocity field $\bar{u}_n$, initial datum $B_{\initial}$ and magnetic diffusivity $\eps >0$.
Let $j_0 \in \NN$, then we slightly abuse the notation and we  denote by  
\begin{align} \label{d:initial-B}
    {B}_{\initial }^{(n)} (\xx) = \e^{i n  z} (-i,1,0) + (0,0,1)\,.
\end{align}
\begin{remark}
    We add the constant vector $(0,0,1)$ to the initial datum, unlike in Section~\ref{sec:ABC}. Although it is a stationary solution to \eqref{passive-vector} for the velocity field $U^{(n)}$ and thus irrelevant here, it will become essential in Section~\ref{sec:construction}. We include it now for convenience.
\end{remark}


\begin{proposition} \label{prop:main}
     There exist $\lambda >0 $, $\delta >0$, $N_0, \bar{n} \in \NN$ such that  the following holds true. For any $n \geq \bar{n}$ sufficiently large there exist $\eta_n, c_n  \in (0, 1/2)$ such that  the solution $B$ to \eqref{passive-vector} with velocity field $U^{(n)}$ defined in \eqref{d:building-block}, initial datum $B_{\initial}$ satisfying $\| B_{\initial} -  {B}_{\initial }^{(n)} \|_{L^2} < \eta_n  \|{B}_{\initial }^{(n)} \|_{L^2} $ and magnetic diffusivity $\eps \in [\frac{1}{ N_0^2 n^2} , \frac{1}{ N_0^2 (n-1)^2} ) $ satisfies 
     $$\| B (t)   \|_{L^2} \geq c_n \e^{\lambda t} \| B_{\initial} \|_{L^2} \qquad \forall t \geq 0 \,.$$
\end{proposition}


The proof of this proposition follows from the following lemma that is a consequence of the analysis carried out in Section \ref{sec:ABC}.

\begin{lemma}
\label{lemma:rescaled 1}
There exist $N_0 \in \NN$, $\delta >0$ and $\lambda>0$\footnote{This $\lambda >0$ is not the same the one in Proposition \ref{prop:main}.} such that the solution $B^{\eps_1}$ to \eqref{passive-vector} with $\eps = \frac{1}{N_0^2} >0$, $u = U_1$ defined in \eqref{d:building-block} with $n=1$, and $B_{\initial}^{(1)} (x,y,z) = \e^{iz} (-i,1,0) + (0,0,1)$ satisfies
$$\| B^{\eps_1} (t) \|_{L^2 (\TT^3)} \geq \frac{1}{2} \e^{\lambda t} \| B_{\initial} \|_{L^2 (\TT^3)} \,, \quad \forall t \geq 0\,. $$
\end{lemma}

\begin{proof}
    We notice that $B^{\eps_1}(t,\xx)$ is a solution to \eqref{passive-vector} with $\eps = \frac{1}{N_0^2}$, velocity field $u (\xx)= U^{(n)} (\xx)$ and initial datum $B_{\initial} (\xx)$ if and only if $B (t, \xx) = B^{\eps_1} (t, \frac{\xx}{N_0})$ solves \eqref{passive-vector} with $\eps =1$, and velocity field 
    $$ N_0 U^{(1)} (\frac{\xx}{N_0}) = \delta \begin{pmatrix}
        \cos(y) 
        \\
        \sin (x) 
        \\
        \sin(y) + \cos (x)
    \end{pmatrix} \,,$$
 and initial datum $B_{\initial} (\frac{\xx}{N_0})$. In particular, choosing $N_0$ sufficiently large so that  $j = \frac{1}{N_0}$ is sufficiently small and $\delta >0$ is sufficiently small we can apply Lemma \ref{lemma:initial growth} and deduce 
 $$ \|B (t) \|_{L^2 ((0,2\pi)^3)} \geq \frac{1}{2} \e^{\lambda t} \| B_{\initial} (\frac{\cdot}{N_0})\|_{L^2 ((0, 2 \pi)^3)} \,, \quad \forall t \geq 0 \,.$$
Noticing that $B$ takes the form $B(t,x,y,z) = e^{i j z}\e^{\mathcal{L} (j) t} (-i,1,0) + (0,0,1)$ and in particular is $2\pi$ periodic in $x,y$, we deduce that 
$$\| B^{\eps_1} (t) \|_{L^2 (\TT^3)} =  \| B (t) \|_{L^2((0, 2 \pi)^3)} \geq  \frac{1}{2} \e^{\lambda t} \| B_{\initial} (\frac{\cdot}{N_0}) \|_{L^2 ((0, 2 \pi)^3)} = \frac{1}{2} \e^{\lambda t} \| B_{\initial}  \|_{L^2 (\TT^3)}  \,,$$
where the last hold true thanks to the explicit formula $| B_{\initial} (\xx) | = \sqrt{3}\,, \forall \xx$.
\end{proof}

\begin{proof}[Proof of Proposition \ref{prop:main}]

By Lemma \ref{lemma:sectorial2} applied to the result from Lemma \ref{lemma:rescaled 1}, there exists some constants $\omega, c_1>0$ such that the following holds true. The solution $B^{(1)}$ to \eqref{passive-vector} with $\eps \in [\frac{1}{N_0^2}-\omega,\frac{1}{N_0^2}+\omega]$, velocity field $U_1$ defined in \eqref{d:building-block} with $n=1$ and initial datum $B_{\initial}^{(1)} = \e^{i z} (-i,1,0) + (0,0,1)$ satisfies 
\begin{equation}
\|B^{(1)}(t)\|_{L^2(\mathbb{T}^3)} \geq c_1 \e^{\frac{\lambda}{2}t} \,, \qquad  \forall t \geq 0 \,.
\end{equation}
Let $B^{(n)} (t,\xx)=B(t,n \xx)$, it follows that $B^{(n)}$ solves 
\begin{equation}
\partial_t B^{(n)} =\nabla \times (U^{(n)} \times B^{(n)})+\frac{\eps}{n^2} \Delta B^{(n)}
\end{equation}
with initial condition $B_{\initial}^{(n)}= \e^{izn}(-i,1,0) + (0,0,1)$. By periodicity on $\TT^3$ we also have $\| B^{(n)} (t) \|_{L^2 (\TT^3)} = \| B^{(n)} (t) \|_{L^2 (\TT^3)} = \| B^{(1)} (t) \|_{L^2 (\TT^3)} $ for all $t \geq 0$.
Hence, we have shown that for any $n \in \mathbb{N}$, $\eps \in [\frac{1}{n^2}(\frac{1}{N_0^2}-\omega),\frac{1}{n^2}(\frac{1}{N_0^2}+\omega)]$, there exists a solution to 
\begin{equation}
\partial_t B=\nabla \times (U^{(n)} \times B)+\eps \Delta B,
\end{equation}
that grows exponentially at rate at least $\frac{\lambda}{2}$. 
Fix now $n \in \mathbb{N}$. For all $\eps \in [\frac{1}{n^2}(\frac{1}{N_0^2}-\omega),\frac{1}{n^2}(\frac{1}{N_0^2}+\omega)]$ we now apply Lemma \ref{lemma:sectorial2} to deduce the existence of positive constants $\eta_\eps ,\omega_\eps, c_\eps >0$ so that the following holds true. The solution to 
\begin{equation}
\partial_t B=\nabla \times (U^{(n)} \times B)+\eps' \Delta B
\end{equation}
with $\eps'$ satisfying $|\eps'-\eps|<\omega_\eps $, initial datum $B_{\initial}$ satisfying  $\|B_{in}-\e^{in z}(-i,1,0) - (0,0,1)\|_{L^2(\mathbb{T}^3)} \leq \eta_\eps $  grows exponentially as
\begin{equation}
\|B(t)\|_{L^2(\mathbb{T}^3)} \geq c_\eps \e^{\frac{\lambda}{4}t} \| B_{\initial} \|_{L^2} \,, \qquad \forall t \geq 0\,.
\end{equation}
Since we can cover the set  $I= [\frac{1}{n^2}(\frac{1}{N_0^2}-\omega),\frac{1}{n^2}(\frac{1}{N_0^2}+\omega)]$ by
\begin{equation}
 I \subset \bigcup_{\eps \in I} (\eps-\omega_\eps , \eps+\omega_\eps ) \,,
\end{equation}
then, by compactness of $I$, we can extract a finite subcover so that 
$$ I \subset  \bigcup_{i=1}^N (\eps_i -\omega_{\eps_i} , \eps_i +\omega_{\eps_i} )  $$
 We define $\eta_n= \frac{1}{\sqrt{2}} \inf_{i=1,\dots N}\eta_{\eps_i}$, $c_n=\inf_{i=1, \dots N}c_{\eps_i}$.
Thus, we have proven that for any $\eps \in [\frac{1}{n^2}(\frac{1}{N_0^2}-\omega),\frac{1}{n^2}(\frac{1}{N_0^2}+\omega)]$ there exist $\eta_n , c_n >0$ such that the following holds true. The  solution $B$ of  
\begin{equation}
\partial_t B=\nabla \times (U^{(n)} \times B)+\eps \Delta B
\end{equation}
with $\eps \in [\frac{1}{n^2}(\frac{1}{N_0^2}-\omega),\frac{1}{n^2}(\frac{1}{N_0^2}+\omega)]$, initial datum $B_{\initial}$ satisfying  $\|B_{in}-B^{(n)}_{\initial}\|_{L^2(\mathbb{T}^3)} < \sqrt{3} \eta_n = \eta_n \| B^{(n)}_{\initial}\|_{L^2} $  grows exponentially as
\begin{equation}
\|B(t)\|_{L^2(\mathbb{T}^3)} \geq c_n \e^{\frac{\lambda}{4}t} \| B_{\initial}\|_{L^2}\,.
\end{equation}
Noticing that $ [\frac{1}{N_0^2 n^2} , \frac{1}{N_0^2 (n-1)^2}] \subset [\frac{1}{n^2}(\frac{1}{N_0^2}-\omega),\frac{1}{n^2}(\frac{1}{N_0^2}+\omega)]$ for any $n \geq \bar{n}$ where $\bar{n}$ depends only on $\omega$, we conclude the proof.
\end{proof}

\section{Construction of the velocity field} \label{sec:construction}

\subsection{Parameters} \label{subsec:paramters}
Firstly, we define the sequence of diffusivity parameters
\begin{equation} \label{d:eps-n}
    \eps_{n} = \frac{1}{N_0^2 n^2} \qquad \forall n \in \NN \,.
\end{equation}
Then, we fix a sequence that visits each $n \in \NN$ infinitely many times. More precisely, we fix 
a sequence $(a_k)_{k \ge 1}$ of natural numbers such that
\[
\forall n \in \mathbb{N}, \quad \#\{ k \in \mathbb{N} : a_k = n \} = \infty,
\]
where $\#$ denotes the cardinality of the set.  
Furthermore,  for every $n \in \mathbb{N}$, we denote by $(k_j)_{j \ge 1}$   an infinite strictly increasing sequence of indices  such that
\[
a_{k_j} = n \quad \text{for all } j \ge 1\,,
\]
where we drop the dependence on $n$ of the sequence $(k_j)_j$ to simplify the notation.

One explicit example is the ``diagonal'' sequence
\[
a_1 = 1, \; a_2 = 1, \; a_3 = 2, \; a_4 = 1, \; a_5 = 2, \; a_6 = 3, \; a_7 = 1, \dots
\]
Finally, we set $t_0 = 0$ and define recursively the sequence of times $(t_k)_{k \ge 1}$ as follows. 
Suppose we have defined $t_0,t_1, \ldots, t_{k-1}$. Let $n \in \NN$ so that $a_k = n $, 
we choose $t_k$ large enough so that
\begin{equation}\label{eq:tk}
    t_k \ge t_{k-1} + 2 n^2 t_{k-1} + 2 n^2 \log\!\left( \frac{10}{\eta_n} \right) + \frac{10}{\lambda} \log(C_n),
\end{equation}
where $\lambda, C_n, \eta_n > 0$ are the constants given in Proposition~\ref{prop:main}.

\subsection{Definition of the velocity field} 

Firstly, we define another building block $U^{(n)}_{g}: \T^3 \to \RR^3$ which is needed to create an initial datum close to  \eqref{d:initial-B} so that we can apply Proposition \ref{prop:main}. For any $n \in \NN$ we simply define the 1-Lipschitz velocity field
\begin{align} \label{d:velocity-generation}
    U^{(n)}_{g}(x,y,z) = \frac{e^{i n  z}}{i n } (-i, 1,0)  \,.
\end{align}

We define the velocity field on each time interval 
 $t \in [t_{k-1}, t_k]$ for any $k \in \NN$. Let $n \in \NN$ so that $a_k = n $. Then, denoting by 
 \begin{equation} \label{d:delta-t}
     \Delta t  = 4 n^2 N_0^2 t_{k-1} + 4 n^2 N_0^2 \log(\frac{10}{\eta_n}) + \frac{4}{\lambda} \log(C_n) \,,
 \end{equation} 
  and  we define 
 \begin{align} \label{d:u}
     u (t, \cdot ) = \begin{cases}
         0 \qquad  & \text{if} \quad t \in [t_{k-1}, t_{k-1} + \frac{ \Delta t}{2}) 
         \\
         U^{(n)}_{g} (\cdot) & \text{if} \quad  t \in [ t_{k-1} + \frac{ \Delta t}{2}, t_{k-1} + \frac{ \Delta t}{2} +1 )  
         \\
         U^{(n)} (\cdot ) & \text{if} \quad  t \in [  t_{k-1} + \frac{ \Delta t}{2} +1 , t_k)  
     \end{cases}
 \end{align}

 \begin{lemma}
     The velocity field defined in \eqref{d:u} satisfies
     $$ \| u (t) \|_{W^{1,\infty} (\TT^3)} \leq 2 \qquad \forall t >0 \,.$$
 \end{lemma}

 \begin{lemma} \label{lemma:time}
     Let $u=U^{(c)}_n \in C^\infty (\TT^3)$ and $B_{\initial} = (0,0,1) + B_{\initial, \text{err}}$ with $\| B_{\initial, \text{err}} \|_{L^2} < \eta_n/10$ and $B$ be the solution to \eqref{passive-vector} with $\eps \in [\frac{1}{ N_0^2n^2} , \frac{1}{ N_0^2 (n-1)^2} ) $ and $n \geq 2$ where $\eps_\star$ is given by Proposition \ref{prop:main}. Then  there exists $\alpha \in [1, 2]$ such that
     $$\| \alpha  B(1) - B_{\initial}^{(n)} \|_{L^2} < \eta_n  \|  B_{\initial }^{(n)}  \|_{L^2} $$
 \end{lemma}

 \begin{proof}
     Let $\tilde{B}$ be the solution to \eqref{passive-vector} with initial condition $\tilde{B} = (0,0,1)$. Then, by explicit computation we have 
     $$ \tilde{B} (t,z) = \varphi (t,z) (-i, 1,0) $$
     where 
     $$ \varphi (t,z) = \int_0^t \e^{- \eps n^2 (t-s)} e^{i nz} ds = \frac{1}{\eps n^2} (1- \e^{- \eps n^2}) \e^{in z} \,.$$
     Since $\eps \in [\frac{1}{ N_0^2 n^2} , \frac{1}{ N_0^2 (n-1)^2} ) $ we deduce that $\alpha^{-1} =\frac{1}{\eps n^2} (1- \e^{- \eps n^2}) \in  [1/2,1]$.
     Hence,  $\alpha \tilde{B} (1) = B_{\initial}^{(n)}$ with $\alpha \in [1,2]$. 

     Finally, we observe that $\| u \|_{W^{1, \infty}} \leq 2$ and by energy estimate the solution $B_{err}$ with initial condition $B_{\initial, err}$ satisfies
     $$\| B_{err} (t) \|_{L^2} \leq e^{t} \| B_{\initial , err} \|_{L^2} \,.$$
     Therefore, we conclude by the linearity of the equation.
 \end{proof}

\section{Proof of Theorem \ref{thm:main}}

We fix $\eps_\star \in (0, \frac{1}{N_0^2 \bar{n}^2})$ where $N_0 \in \NN$ and $\bar{n} \geq 2$ are natural numbers so that Proposition \ref{prop:main} applies.
We define the initial condition 
$$ B_{\initial } = (0,0,1) $$
and we denote by $B$ the solution of the passive vector equation \eqref{passive-vector} with velocity field $u \in L^\infty $ and diffusivity parameter $\eps \in (0, \eps_\star)$.
Since $\eps \in (0, \eps_\star)$, there exists $n \in \NN$ so that $\eps \in [\frac{1}{N_0^2 n^2}, \frac{1}{N_0^2 (n-1)^2} )$, where $n \geq 2$ is sufficiently large so that  Proposition \ref{prop:main} applies.  We now fix an increasing subsequence $( k_j )_{j} \subset \NN$ such that the sequence $(a_k)_k$  defined in Section \ref{subsec:paramters} satisfies
$$ a_{k_j} = n \qquad \forall j \,. $$

Then, for the sequence of times $( t_{k})_{k \in \NN}$  defined in Section \ref{subsec:paramters} we claim that 
\begin{equation} \label{clam}
\| B (t_{k_j}) \|_{L^2} \geq \e^{\sfrac{\lambda t_{k_j} }{4}} \| B_{\initial} \|_{L^2}  \qquad \forall j \in \NN \tag{claim}
\end{equation}
from which we deduce the theorem, since $\lambda >0$ given by Proposition \ref{prop:main}. Let $j \in \NN$ be arbitrary. By standard energy estimates using that $\| \nabla u \|_{L^\infty} \leq 2$ we have that 
$$\|B (t_{k_j -1}) \|_{L^2}  \leq \e^{ t} \| B_{\initial} \|_{L^2} \,.$$
Thanks to $u \equiv 0$ on $t \in [t_{k_j -1} , t_{k_j -1} + \Delta t/2)$ by definition \eqref{d:u} and the fact that the average of $B$ is conserved in time we conclude by the heat semigroup property that 
$$  \| B(t) -(0,0,1)  \|_{L^2} \leq \e^{- \eps (t - t_{k_j -1})/2} \| B (t_{k-1}) - (0,0,1) \|_{L^2} \quad \forall t \in  [t_{k_j -1} , t_{k_j -1} + \frac{\Delta t}{2})\,.$$
Thanks to
$\eps \geq \frac{1}{N_0^2 n^2}$
and  $\Delta t \geq  {4 N_0^2 n^2} t_{k-1} + {4 N_0^2 n^2} \log(\frac{10}{\eta_n})$ we deduce
$$\| B(t_{k_j -1 } +  \frac{\Delta t}{2} ) - (0,0,1) \|_{L^2} \leq \frac{\eta_n}{10} \,.  $$
By definition of $u$ on $t \in [t_{k_j -1} + \frac{\Delta t}{2}, t_{k_j -1} + \frac{\Delta t}{2} + 1 )$ and Lemma \ref{lemma:time}, recalling the notation for $B_{\initial , n}^{(p)}$ in \eqref{d:initial-B} we deduce that there exists $\alpha \in [1,2]$ such that
$$\| \alpha B(t_{k_j -1} + \frac{\Delta t}{2} + 1) -  B_{\initial  }^{(n)} \|_{L^2} < \eta_n  \| B_{\initial }^{(n)} \|_{L^2} \,.$$
Finally, using the linearity of the equation we apply Proposition \ref{prop:main} and using the bound on $t_{k_j}$ given by \eqref{eq:tk} we deduce that 
\begin{align}
    \| B(t_{k_j}) \|_{L^2} & \geq \frac{C_n}{\alpha} \e^{\lambda (t_{k_j} - (t_{k_j -1} + \frac{\Delta t}{2} + 1))} \|B(t_{k_j -1} + \frac{\Delta t}{2} + 1) \|_{L^2} 
    \\
    &\geq \frac{C_n}{4} \e^{\lambda t_{k_j}/3} \| B_{\initial }^{(n)} \|_{L^2} 
    \\
    & \geq \e^{\lambda t_{k_j}/4}   \| B_{\initial} \|_{L^2} \,.
\end{align}

\appendix

\section{Proofs on sectorial operators} \label{sec:appendix}
In this Appendix we prove the results from Section \ref{sec:sectorial}.

\begin{proof}[Proof of Lemma \ref{lemma:sectorial1}]
To aid in readability, we shall denote the resolvent $R_z=(z-T)^{-1}$. By standard analytic Semigroup theory, we for any $b > b_0$, $\theta \in (0,\theta_0)$, the contour $\gamma_{\theta,b}$ given by the boundary of the sector $\{Arg(b-z) \in [-\frac{\pi}{2}+\theta,\frac{\pi}{2}-\theta]\}$ satisfies 
\begin{equation}
\e^{T t}=\frac{1}{2 \pi i}\int_{\gamma_{\theta,b}}(z-T)^{-1} \dd z.
\end{equation}
Now, take $p<b$. Then, for all $\zeta>0$ small enough, the strip $\{\Re(z) \in [p-\zeta,p-\frac{\zeta}{2}]\}$ does not intersect the spectrum of $T$. Indeed, suppose this was not true. Then, there would exist a sequence of eigenvalues $\lambda_n$ of $T$, with $\Re(\lambda_n) \in [p-\frac{1}{3^n},p-\frac{1}{2\times 3^n }]$. In particular, note that all the $\lambda_n$ must be distinct, since the intervals $[p-\frac{1}{3^n},p-\frac{1}{2\times 3^n }]$ are non-overlapping. Furthermore, due to the sectorial property, the $\lambda_n$ have uniformly bounded imaginary part. Thus, there exists a (non-relabeled) convergent subsequence, converging to some $\lambda \in \mathbb{C}$, $\Re(\lambda)=p$. By closedness of the spectrum, $\lambda \in \sigma(T)$, and clearly $\lambda$ is an accumulation point of the spectrum, which cannot occur due to the compactness of the resolvent of $T$. Therefore, for all $\zeta$ small enough, we can construct new contours $\gamma^1_{\theta,b}$, $\gamma^2_{\theta,b}$, which are constructed as follows. $\gamma^1_{\theta,b}$ proceeds by first traversing $\gamma_{\theta,b}$ along its lower branch, until the line $\{\Re(z)=p-\frac{\zeta}{2}\}$ is reached. Then, traverse the line $\{\Re(z)=p-\frac{\zeta}{2}\}$ upwards until it intersects the upper branch of $\gamma_{\theta,b}$, at which point one traverses this upper branch towards the left. Similarly, $\gamma^2_{\theta,b}$ is obtained by first traversing the contour $\gamma_{\theta,b} \cap \{\Re(z) \geq p-\frac{\zeta}{2}\}$, and then traversing the line $\{\Re(z) =p-\frac{\zeta}{2}\}$ downwards in order to close the contour. Overall, since the traversals along $\{\Re(z) =p-\frac{\zeta}{2}\}$ of the contour cancel, we immediately obtain 
\begin{equation}
\e^{T t}=\frac{1}{2 \pi i}\int_{\gamma^1_{\theta,b}}R_z \e^{zt} \dd z+\frac{1}{2 \pi i}\int_{\gamma^2_{\theta,b}}R_z \e^{zt} \dd z.
\end{equation}
But note now that since $\{\Re(z) \in [p-\zeta, p-\frac{\zeta}{2}]$ contains no spectral points of $T$, by Cauchy's theorem we may deform $\gamma^1_{\theta,b}$ to traverse the line $\{\Re(z)=p-\zeta\}$ upwards instead of the line $\{\Re(z)=p-\frac{\zeta}{2}\}$. Hence, setting $\gamma_1=\gamma^1_{\theta,b}$ (the deformed version), and $\gamma_2=\gamma^2_{\theta,b}$, these satisfy the claimed properties.

Next, we prove the claimed growth bounds. To begin, note that $P$ from \eqref{eq:P abstract} is well defined, and is nothing but the Riesz projector onto the space spanned by eigenvalues of $T$ with real part at least $p-\frac{\zeta}{2}$. By the functional calculus, it holds that $\e^{T t} P=S_2(t)$. In fact, we furthermore deduce that that $S_2(t)$ in fact generates a group. In particular, the integral expression defining $S_2(t)$ is well defined if $t<0$, and by the functional calculus it follows that $S(-t)S(t)=P$. Furthermore, observe the estimate 
\begin{equation}
\|S_2(-t)\| \leq \frac{1}{2\pi}\mathrm{len}(\gamma_2)\sup_{z \in \gamma_2}\|R_z\| \e^{-t(p-\frac{\zeta}{2})}.
\end{equation}
Therefore, we deduce that 
\begin{equation}
\|Px\|=\|S(-t) S(t)x\|\leq \frac{1}{2\pi}\mathrm{len}(\gamma_2)\sup_{z \in \gamma_2}\|R_z\| \e^{-t(p-\frac{\zeta}{2})}\|S(t)x\|,
\end{equation}
and thus 
\begin{equation}
\|S_2(t)x\|\geq C_1(\gamma_2,\sup_{z \in \gamma_2}\|R_z\|)e^{t(p-\frac{\zeta}{2})}\|Px\|.
\end{equation}
Next, we proceed to bounding $S_1(t)$. To do so, we identify the contour $\gamma_{\theta,b}$ with the lines $\ell_1=\{x(1+i\alpha)-i\alpha b, x \leq b\}$, $\ell_2=\{x(1-i\alpha)+i\alpha b, x \leq b\}$, for some $\alpha>0$. Then, $S_1(t)$ is nothing but 
\begin{align}
&S_1(t)=\frac{1}{2 \pi i}\int_{-\infty}^{p-\zeta} R_{x(1+i\alpha)-i\alpha b} \e^{t(x(1+i\alpha)-i\alpha b)}(1+i\alpha) \dd x +\frac{1}{2 \pi i}\int_{\alpha(p-\zeta-b)}^{\alpha(b-p+\zeta)}iR_{p-\zeta+ix}\e^{t(p-\zeta)} \dd x\\
&+\frac{1}{2 \pi i}\int_{p-\zeta}^{-\infty} R_{x(1-i\alpha)+i\alpha b} \e^{t(x(1+i\alpha)-i\alpha b)}(1-i\alpha) \dd x.
\end{align}
Hence, bounding these terms we obtain
\begin{equation}
\|S_1(t)\|\leq C_2(\gamma_1, \sup_{z \in \gamma_1}\|R_z\|)(1+t^{-1})\e^{t(p-\zeta)}.
\end{equation}
Hence, overall we observe that if $Px=0$, it holds 
\begin{equation}
\|S(t)x\|\leq C_2(\gamma_1, \sup_{z \in \gamma_1}\|R_z\|)(1+t^{-1})\e^{t(p-\zeta)}\|x\|,
\end{equation}
whereas in general we obtain the lower bound 
\begin{equation}
\|S(t)x\|\geq C_1(\gamma_2, \sup_{z \in \gamma_2}\|R_z\|)e^{t(p-\frac{\zeta}{2})}\|Px\|-C_2(\gamma_1, \sup_{z \in \gamma_1}\|R_z\|)(1+t^{-1})\e^{t(p-\zeta)}\|x\|,
\end{equation}
and so the claim of the Lemma follows.
\end{proof}

\medskip

\begin{proof}[Proof of Lemma \ref{lemma:sectorial2}]
Upon replacing $x_0 \mapsto \frac{x_0}{\|x_0\|}$ we can assume that $\|x_0\|=1$ and deduce the general result by linearity. Note that there holds for $\lambda \notin (-\infty, 0]$
\begin{align}
&\|T_1(\kappa T_0-\lambda)^{-1}\|\leq \delta\kappa \|T_0(\kappa T_0-\lambda)^{-1}\|+C(\delta \kappa) \|(\kappa T_0-\lambda)^{-1}\|\\
&\leq \delta +\delta |\lambda| \|(\kappa T_0-\lambda)^{-1}\|+C(\delta \kappa) \|(\kappa T_0-\lambda)^{-1}\|\\
&{ \leq} \delta (1+\frac{|\lambda|}{\dist(\lambda, (-\infty,0])})+\frac{C(\delta \kappa)}{\dist(\lambda, (-\infty,0])}.
\end{align}
Thus, there exists some $b_0>0$, $0<\theta_0<\frac{\pi}{2}$ so that for all $\lambda \notin S_{\theta_0,b_0}$, $ \kappa \in (\frac{1}{2},\frac{3}{2})$, it holds $|\lambda|\leq 2\dist(\lambda, (-\infty,0])$, and $\frac{C( \frac{\kappa}{6})}{\dist(\lambda, (-\infty,0])}\leq \frac{1}{3}$. Thus, there holds 
\begin{equation}
\|T_1(\kappa T_0-\lambda)^{-1}\|\leq \frac{5}{6}
\end{equation}
uniformly for $\lambda \notin S_{\theta_0,b_0}$, $\kappa \in (\frac{1}{2},\frac{3}{2})$. Hence, we have 
\begin{equation}
\|(T(\kappa)-\lambda)^{-1}\|=\|(\kappa T_0-\lambda)^{-1}(\Id+T_1(\kappa T_0-\lambda)^{-1})^{-1}\|\leq \frac{5}{6\dist(\lambda, (-\infty,0])},
\end{equation}
which is the analytic semigroup bound. Therefore, for all $\kappa \in (\frac{1}{2}, \frac{3}{2})$, the semigroup generated by $T(\kappa)$ is given by 
\begin{equation}
\e^{T(\kappa)t}=\frac{1}{2 \pi i}\int_{\gamma}(z-T(\kappa))^{-1} \e^{z t} \dd z,
\end{equation}
where the contour $\gamma$ is independent of $\kappa$, and is the boundary of some sector $S_{\theta,b}$ with $\theta<\theta_0$, $b >b_0$. Then, by assumption and Lemma \ref{lemma:sectorial1}, we may write 
\begin{equation}
\e^{T(1)t}=\frac{1}{2 \pi i}\int_{\gamma_1}(z-T(1))^{-1} \e^{zt} \dd z+\frac{1}{2 \pi i}\int_{\gamma_2}(z-T(1))^{-1}\e^{zt} \dd z,
\end{equation}
where $\gamma_1$, $\gamma_2$ are constructed from $\gamma$ as in the proof of Lemma \ref{lemma:sectorial1}. In particular, $\gamma_1 \subset \{\Re(z) \leq p-\zeta\}$ for some $\zeta>0$ and $\gamma_2 \subset \{\Re(z) \geq p-\frac{\zeta}{2}\}$. Now, we claim this decomposition is valid for all $\kappa$ sufficiently close to $1$. Indeed, from the proof of Lemma \ref{lemma:sectorial1}, it is clear that the decomposition follows as soon as $\sigma(T(\kappa)) \cap \{\Re(z) \in [p-\zeta, p-\frac{\zeta}{2}]\} = \emptyset$. But now, we compute 
\begin{equation}
(\kappa T_0+T_1-\lambda)=(\Id+(\kappa-1)T_0 (T_0+T_1-\lambda)^{-1})( T_0+T_1-\lambda).
\end{equation}
As soon as $\|(\kappa-1)T_0 ( T_0+T_1-\lambda)^{-1}\|<1$ this is invertible. But now,
\begin{equation}
\|T_0 (T_0+T_1-\lambda)^{-1}\| \leq 1+\|(T_1-\lambda)(T_0+T_1-\lambda)^{-1}\|\leq 1+|\lambda|\|(T_0+T_1-\lambda)^{-1}\|+\|T_1(T_0+T_1-\lambda)^{-1}\|.
\end{equation}
But by assumption this last term may be bounded by 
\begin{equation}
\frac{1}{2}\|T_0(T_0+T_1-\lambda)^{-1}\|+C(\frac{1}{2})\|(T_0+T_1-\lambda)^{-1}\|. 
\end{equation}
Hence, we observe that 
\begin{equation}
\|T_0 (T_0+T_1-\lambda)^{-1}\|  \leq 2\left (1+(|\lambda|+C(\frac{1}{2})\|(T_0+T_1-\lambda)^{-1}\| \right).
\end{equation}
Taking suprema over $\lambda \in S_{\theta,b} \cap \{\Re(z) \in [p-\zeta, p-\frac{\zeta}{2}]\}$, we observe from the continuity of the resolvent with respect to $\lambda $ that 
\begin{equation}
\sup_{\lambda \in S_{\theta,b} \cap \{\Re(z) \in [p-\zeta, p-\frac{\zeta}{2}]\}}\|(T_0(T_0+T_1-\lambda)^{-1}\|  \leq K,
\end{equation}
for some constant $K$.
Thus, $\{\Re(z) \in [p-\zeta, p-\frac{\zeta}{2}]\} \subset \rho(T(\kappa))$ if $|\kappa-1|\leq \frac{1}{2}K^{-1}$, and we have the uniform resolvent bound for all $\lambda$ in this strip 
\begin{equation}
\|(T(\kappa)-\lambda)^{-1}\|\leq C(K). 
\end{equation}
Furthermore, combined with the estimates at the beginning of the proof, we deduce that $\sup_{z \in \gamma_1}\|(z-T(\kappa))^{-1}\|+\sup_{z \in \gamma_2}\|(z-T(\kappa))^{-1}\|\leq C(K)$ for all $|\kappa-1|\leq \frac{1}{2}K^{-1}$. Thus, uniformly for $\kappa$ in this interval, it holds 
\begin{equation}
\|\e^{T(\kappa)t} x\|\geq C_1 \e^{t(p-\frac{\zeta}{2})}\|P(\kappa)x\|-C_2(1+t^{-1})\e^{t(p-\zeta)}\|x\|,
\end{equation}
where 
\begin{equation}
P(\kappa)=\frac{1}{2 \pi i}\int_{\gamma_2}(z-T(\kappa))^{-1} \\d z.
\end{equation}
Now, by the same estimates as above, it follows that $\|P(\kappa)-P(1)\| \to 0$ as $\kappa \to 1$. Furthermore, by Lemma \ref{lemma:sectorial1} it holds that $\|Px_0\|>0$. Therefore, we have 
\begin{equation}
\|P(\kappa)x\|=\|(P(\kappa)-P)x+Px\|\geq \|P x_0\|-\|P(x-x_0)\|-\|(P-P(\kappa))x\|.
\end{equation}
Thus, so long as $\|P-P(\kappa)\|\leq \frac{\|Px_0\|}{3\|x\|}$, $\|x-x_0\|\leq \frac{\|Px_0\|}{3\|P\|}$, there holds 
\begin{equation}
\|P(\kappa)x\| \geq \frac{\|Px_0\|}{3}.
\end{equation}
Hence, this is true for all $\|x-x_0\|\leq \eta_0 $, for some $\eta_0>0$, and for $|\kappa-1|\leq \kappa_0$ for some $\kappa_0>0$. Thus, for all such $x, \kappa $ it holds 
\begin{equation}
\|\e^{T(\kappa)t} x\|\geq \frac{1}{3}C_1 \e^{t(p-\frac{\zeta}{2})}\|Px_0\|-2C_2(1+t^{-1})\e^{t(p-\zeta)}\|x_0\|.
\end{equation}
Now, since $\|Px_0\| \neq 0$, there exists a constant so that $\|Px_0\|\geq c\|x_0\|$, and so for all $t$ large enough, it holds 
\begin{equation}
\|\e^{T(\kappa)t} x\| \geq C\e^{t(p-\frac{3\zeta}{4})}\|x_0\|.
\end{equation}
We can further note (upon possibly reducing the size of $\eta$), that $\|x_0\|\geq \frac{1}{2}\|x\|$ for all $\|x-x_0\|\leq \eta_0$, and so the claim of the Lemma follows.
\end{proof}

\begin{proof}[Proof of Lemma \ref{lemma:kato-convergence}]
The first half of the proof is exactly as in \cite{CZSV25}. We recount it here since we require some notation.
Let $\Gamma$ be a curve around the eigenvalue $p$, and let $P, P(j)$ denote the Riesz projectors associated with the operators $T_0, T(j)=T_0+j T_1$, respectively, around $\Gamma$. Denote the resolvents of $T_0$ and $T(j)$ by $R(\mu;T_0) = (T_0 - \mu)^{-1}$ and $R(\mu;T(j)) = (T(j) - \mu)^{-1}$, respectively. We begin by considering the operator 
\begin{equation}
V(j)=I-P+P(j)P.
\end{equation}
Note that $V(j)P=P(j)P$ and $V(j)(I-P)=I-P$. By Lemma A.3 from \cite{CZSV25}, there exists some constant $c>0$ such that $\|P(j)-P\| \leq c j$, for $j$ sufficiently. Consequently, for small $j,$ $V(j)$ is invertible, and its inverse can be written as 
\begin{equation}
V(j)^{-1}=I+(P-P(j))P+o(j)
\end{equation}
where  $o(j)$ denotes any operator with norm of order $o(j)$ as $j \to 0$. In particular, $V(j)$ is an isomorphism from $\Ran(P) \to \Ran(P(j))$. 

Next, consider the operator 
\begin{equation}
R_1(\mu;j):=V(j)^{-1}R(\mu;T(j))V(j)P, \qquad \text{for } \mu\in \Gamma.
\end{equation}
This operator annihilates $\Ran(I-P)$. Since $V(j)$ maps $\Ran(P)$ into $\Ran(P(j))$  and  $R(\mu;T(j))$ commutes with $P(j)$ by functional calculus, we conclude that 
$R_1(\mu;j)$ maps $\Ran(P(j))$ into itself. Furthermore, by the invertibility of $V(j)$, we have $R_1(\mu;j)=P R_1(\mu;j)$.
We  now expand $R_1(\mu;j)$ as
\begin{equation}
R_1(\mu;j)=(I+(P-P(j))P+o(j))(R(\mu;T_0)-j R(\mu;T_0) T_1 R(\mu;T_0)+o(j))(I-(P-P(j))P). 
\end{equation}
This simplifies to 
\begin{equation} \label{eq:R-1-mu-kappa}
R_1(\mu;j)=R(\mu;T_0)-j R(\mu;T_0)T_1R(\mu;T_0)-R(\mu;T_0)(P-P(j))P+(P-P(j))P R(\mu;T_0)+o(j). 
\end{equation}
Using the fact that $R(\mu;T_0) P = (p - \mu)^{-1} P$ 
due to the semisimplicity of  $p$ for $T_0$, we get 
\begin{equation}
P R(\mu;T_0)(P-P(j))P=  P (P-P(j))P R(\mu;T_0) P.
\end{equation}
Thus, using that  $R_1(\mu;j)=P R_1(\mu;j)P$ and again $R(\mu;T_0) P = (p - \mu)^{-1} P$, we deduce from  \eqref{eq:R-1-mu-kappa} that
\begin{align*}
R_1(\mu;j)=(p -\mu)^{-1}P-j (p-\mu)^{-2}P T_1P+o(j) \,.
\end{align*}
Multiplying by $\frac{-\mu}{2\pi i}$ and integrating around the contour $\Gamma$, we obtain the expression
\begin{equation}
\label{eq:perturbative expression}
V(j)^{-1}T(j)P(j) V(j)P=p P+j P T_1P +o(j)
\end{equation}
where we have used the fact that  
\begin{equation}
T(j)P (j)=-\frac{1}{2\pi i}\int_{\Gamma}R(\mu;T(j))\mu \dd\mu \,.
\end{equation}
Observe that the $p_i(j)'s$ are the eigenvalues of $T(j)P (j)$ in the $m$-dimensional space $\Ran(P (j))$. By similarity, they are also the eigenvalues of $V(j)^{-1}T(j)P (j)V(j)$ in the $m$-dimensional space $\Ran(P)$. Since $P $ acts as the identity on $\Ran(P)$, they are also equal to the eigenvalues of $V(j)^{-1}T(j)P (j)V(j)P $ on $\Ran(P)$. 
Finally, applying Theorem A.5 from \cite{CZSV25} and using the asymptotic expression \eqref{eq:perturbative expression}, the first claim follows immediately.
Now, note that $P_i(j)$ is the one-dimensional eigenprojection for the eigenvalue $\mu_i+o_{j \to 0}(1)$ of the operator 
\begin{equation}
\tilde{T}^{(1)}(j):=j^{-1}(T(j)-p)P(j):P(j)X \to P(j) X.
\end{equation}
By \eqref{eq:perturbative expression} we have that 
\begin{equation}
\label{eq:kato perturbation 2}
V(j)^{-1}\tilde{T}^{(1)}(j)P(j) V(j)P= P T_1P +o_{j \to 0}(1),
\end{equation}
since $V(j)^{-1} P(j)V(j)P=P$, as $V(j)$ maps $Px \mapsto P(j)x.$ Note also that $V(j)^{-1}\tilde{T}^{(1)}(j)P(j) V(j)P $ is an operator from $PX \mapsto PX$. Thus, since $PT_1 P:PX \to PX$ has $m$ distinct eigenvalues, for all $j$ small enough it follows from \eqref{eq:kato perturbation 2} that the same holds for $V(j)^{-1}\tilde{T}^{(1)}(j)P(j) V(j)P$. Choose its $m$ eigenvectors to be $\psi_{i}(j)$, $i=1, \dots m$. From \eqref{eq:kato perturbation 2} it follows that $\psi_i(j) \to \psi_i$, where $\psi_i$ is the eigenvector corresponding to an eigenvalue $\mu_i$ of $PT_1P$. Next, note that in fact $\phi_i(j)=V(j) \psi_i(j)$ is an eigenvector of $\tilde{T}^{(1)}(j)$ with eigenvalue $\mu_i(j)$. Indeed, from \eqref{eq:kato perturbation 2} we see that, since $\psi_i(j) \in PX$, it holds 
\begin{equation}
\tilde{T}^{(1)}(j)V(j) \psi_i(j)=\tilde{T}^{(1)}(j)V(j) P\psi_i(j) = \mu_i V(j)\psi_i(j).
\end{equation}
Furthermore, since $V(j) \to \Id$ as $j \to 0$, it follows that $\phi_i(j) \to \psi_i$ as $j \to 0$.
Since the $\phi_i(j)$ are linearly independent, they form a basis for the $m$-dimensional space $P(j)X$. Thus, given $x \in X$ it holds
\begin{equation}
P(j)x=\sum_{i=1}^m \alpha_i(j)\phi_i(j),
\end{equation}
for some uniquely determined coefficients $\alpha_i(j)$. From this it follows that $P_i^{(1)}(j) x=\alpha_i(j)\psi_i(j)$, $i=1, \dots m$. Furthermore, we claim that the $\alpha_i(j)$ converge as $j \to 0$. 
Indeed, consider the following map $M(j) :PX \to PX$, defined by $M(j)\phi_i =\psi_i(j)$. Since $\psi_i(j) \to \phi_i$, it follows that $M(j) \to \Id$ as $j \to 0$. Therefore, $M^{-1}(j) \to \Id$ as well, and it holds
\begin{equation}
M^{-1}(j)V(j)^{-1}P(j)x=M^{-1}(j)\sum_{i=1}^m \alpha_i(j)\psi_i(j)=\sum_{i=1}^m\alpha_i(j)\phi_i.
\end{equation}
The left-hand side converges to $Px$ as $j \to 0$, and so, since the $\phi_i$ form a basis of $PX$, it follows that the $\alpha_i(j)$ must converge to some limit $\alpha_i$. Therefore, the projectors $P_i^{(1)}(j)x=\alpha_i(j)\psi_i(j)$ converge pointwise to $\alpha_i \phi_i$. But note that, since 
\begin{equation}
Px=\sum_{i=1}^m \alpha_i \phi_i,
\end{equation}
it follows that $P^{(1)}_iPx=\alpha_i \phi_i$, and so the proof is completed.
\end{proof}

\addtocontents{toc}{\protect\setcounter{tocdepth}{0}}
\section*{Acknowledgments}
The authors are grateful to Michele Coti Zelati for stimulating and insightful discussions.
MS acknowledges support from the Chapman Fellowship at Imperial College London. The research of DV was funded by the Imperial College President’s PhD Scholarships.

\addtocontents{toc}{\protect\setcounter{tocdepth}{1}}

 \bibliographystyle{abbrv}
 \bibliography{biblio.bib}

\begin{bibdiv}
\begin{biblist}

\bib{Albritton_Brué_Colombo_2022}{article}{
      author={Albritton, Dallas},
      author={Bru\'{e}, Elia},
      author={Colombo, Maria},
       title={Non-uniqueness of {L}eray solutions of the forced {N}avier-{S}tokes equations},
        date={2022},
        ISSN={0003-486X},
     journal={Ann. of Math. (2)},
      volume={196},
      number={1},
       pages={415\ndash 455},
         url={https://doi.org/10.4007/annals.2022.196.1.3},
}

\bib{Arnold04}{book}{
      author={Arnold, Vladimir~I.},
       title={Arnold's problems},
   publisher={Springer-Verlag, Berlin; PHASIS, Moscow},
        date={2004},
        ISBN={3-540-20614-0},
        note={Translated and revised edition of the 2000 Russian original, With a preface by V. Philippov, A. Yakivchik and M. Peters},
}

\bib{AK98}{book}{
      author={Arnold, Vladimir~I.},
      author={Khesin, Boris~A.},
       title={Topological methods in hydrodynamics},
      series={Applied Mathematical Sciences},
   publisher={Springer-Verlag, New York},
        date={1998},
      volume={125},
        ISBN={0-387-94947-X},
}

\bib{BaxendaleRozovskii93}{article}{
      author={Baxendale, P.~H.},
      author={Rozovskii, B.~L.},
       title={Kinematic dynamo and intermittence in a turbulent flow},
        date={1993},
     journal={Geophysical \& Astrophysical Fluid Dynamics},
      volume={73},
      number={1-4},
       pages={33\ndash 60},
}

\bib{BBPS22}{article}{
      author={Bedrossian, Jacob},
      author={Blumenthal, Alex},
      author={Punshon-Smith, Sam},
       title={Lagrangian chaos and scalar advection in stochastic fluid mechanics},
        date={2022},
        ISSN={1435-9855,1435-9863},
     journal={J. Eur. Math. Soc. (JEMS)},
      volume={24},
      number={6},
       pages={1893\ndash 1990},
         url={https://doi.org/10.4171/jems/1140},
}

\bib{brandenburg2005astrophysical}{article}{
      author={Brandenburg, Axel},
      author={Subramanian, Kandaswamy},
       title={Astrophysical magnetic fields and nonlinear dynamo theory},
        date={2005},
     journal={Physics Reports},
      volume={417},
      number={1-4},
       pages={1\ndash 209},
}

\bib{Chicone_Latushkin_Montgomery-Smith_1995}{article}{
      author={Chicone, C.},
      author={Latushkin, Y.},
      author={Montgomery-Smith, S.},
       title={The spectrum of the kinematic dynamo operator for an ideally conducting fluid},
        date={1995},
        ISSN={0010-3616},
     journal={Comm. Math. Phys.},
      volume={173},
      number={2},
       pages={379\ndash 400},
         url={http://projecteuclid.org/euclid.cmp/1104274734},
}

\bib{CG95}{book}{
      author={Childress, Stephen},
      author={Gilbert, Alan},
       title={Stretch, twist, fold: The fast dynamo},
   publisher={Springer-Verlag},
     address={Berlin},
        date={1995},
        ISBN={978-3-540-11686-2},
}

\bib{CZNF24}{article}{
      author={{Coti Zelati}, Michele},
      author={{Navarro-Fern{\'a}ndez}, V{\'\i}ctor},
       title={{Three-dimensional exponential mixing and ideal kinematic dynamo with randomized ABC flows}},
        date={2024-07},
     journal={arXiv:2407.18028},
}

\bib{Cowling33}{article}{
      author={Cowling, T.~G.},
       title={The magnetic field of sunspots},
        date={1933},
        ISSN={0035-8711},
     journal={Mon. Not. R. Astron. Soc.},
      volume={94},
       pages={39\ndash 48},
}

\bib{FV91}{article}{
      author={Friedlander, Susan},
      author={Vishik, Misha~M.},
       title={Dynamo theory, vorticity generation, and exponential stretching},
        date={1991},
        ISSN={1054-1500},
     journal={Chaos},
      volume={1},
      number={2},
       pages={198\ndash 205},
         url={https://doi.org/10.1063/1.165829},
}

\bib{Gerardvaret05}{article}{
      author={G\'{e}rard-Varet, David},
       title={Oscillating solutions of incompressible magnetohydrodynamics and dynamo effect},
        date={2005},
        ISSN={0036-1410},
     journal={SIAM J. Math. Anal.},
      volume={37},
      number={3},
       pages={815\ndash 840},
         url={https://doi.org/10.1137/S0036141004444603},
}

\bib{GR07}{article}{
      author={G\'{e}rard-Varet, David},
      author={Rousset, Fr\'{e}d\'{e}ric},
       title={Shear layer solutions of incompressible {MHD} and dynamo effect},
        date={2007},
        ISSN={0294-1449},
     journal={Ann. Inst. H. Poincar\'{e} C Anal. Non Lin\'{e}aire},
      volume={24},
      number={5},
       pages={677\ndash 710},
         url={https://doi.org/10.1016/j.anihpc.2006.04.005},
}

\bib{Gilbert88}{article}{
      author={Gilbert, Andrew~D.},
       title={Fast dynamo action in the {Ponomarenko} dynamo},
        date={1988},
        ISSN={0309-1929},
     journal={Geophys. Astrophys. Fluid Dyn.},
      volume={44},
      number={1-4},
       pages={241\ndash 258},
}

\bib{K76}{book}{
      author={Kato, Tosio},
       title={Perturbation theory for linear operators},
     edition={Second},
      series={Grundlehren der Mathematischen Wissenschaften},
   publisher={Springer-Verlag, Berlin-New York},
        date={1976},
      volume={Band 132},
}

\bib{Klapper_Young_1995}{article}{
      author={Klapper, I.},
      author={Young, L.~S.},
       title={Rigorous bounds on the fast dynamo growth rate involving topological entropy},
        date={1995Nov},
     journal={Communications in Mathematical Physics},
      volume={173},
      number={3},
       pages={623–646},
}

\bib{larmor1919possible}{article}{
      author={Larmor, Joseph},
       title={Possible rotational origin of magnetic fields of sun and earth},
        date={1919},
     journal={Elec. Rev},
      volume={85},
       pages={512},
}

\bib{Moffatt78}{book}{
      author={{Moffatt}, H.~K.},
       title={{Magnetic field generation in electrically conducting fluids}},
   publisher={Cambridge University Press},
        date={1978},
}

\bib{NF_DV2025nonlinear}{article}{
      author={Navarro-Fern{\'a}ndez, V{\'\i}ctor},
      author={Villringer, David},
       title={Nonlinear instability for the 3d mhd equations around the taylor-couette flow},
        date={2025},
     journal={arXiv preprint arXiv:2510.12707},
}

\bib{NF_DV2025spectral}{article}{
      author={Navarro-Fern{\'a}ndez, V{\'\i}ctor},
      author={Villringer, David},
       title={Spectral instability in the smooth ponomarenko dynamo},
        date={2025},
     journal={arXiv preprint arXiv:2509.19201},
}

\bib{R25}{article}{
      author={Rowan, Keefer},
       title={A subsequentially fast dynamo on $\mathbb{T}^3$},
        date={2025},
     journal={arXiv:2505.23936},
      eprint={2505.23936},
         url={https://arxiv.org/abs/2505.23936},
}

\bib{rudiger2006magnetic}{book}{
      author={R{\"u}diger, G{\"u}nther},
      author={Hollerbach, Rainer},
       title={The magnetic universe: geophysical and astrophysical dynamo theory},
   publisher={John Wiley \& Sons},
        date={2006},
}

\bib{Vishik86}{incollection}{
      author={Vishik, M.~M.},
       title={The periodic dynamo},
        date={1986},
   booktitle={Mathematical methods in seismology and geodynamics, {N}o. 19 ({R}ussian)},
   publisher={``Nauka'', Moscow},
       pages={186\ndash 215, 221, 224},
}

\bib{Vishik89}{article}{
      author={Vishik, M.~M.},
       title={Magnetic field generation by the motion of a highly conducting fluid},
        date={1989},
        ISSN={0309-1929},
     journal={Geophys. Astrophys. Fluid Dynam.},
      volume={48},
      number={1-3},
       pages={151\ndash 167},
         url={https://doi.org/10.1080/03091928908219531},
}

\bib{CZSV25}{article}{
      author={Zelati, Michele~Coti},
      author={Sorella, Massimo},
      author={Villringer, David},
       title={Alpha-unstable flows and the fast dynamo problem},
        date={2025},
      eprint={2504.00855},
         url={https://arxiv.org/abs/2504.00855},
}

\bib{Zeldovich_1992}{article}{
      author={Zeldovich, Yakov~Borisovich},
       title={7. a magnetic field in the two-dimensional motion of a conducting turbulent fluid},
        date={1992Dec},
     journal={Selected Works of Yakov Borisovich Zeldovich, Volume I},
       pages={93–96},
}

\bib{ZRMS84}{article}{
      author={Zeldovich, Yakov~Borisovich},
      author={Ruzmaikin, A.~A.},
      author={Molchanov, S.~A.},
      author={Sokoloff, D.~D.},
       title={Kinematic dynamo problem in a linear velocity field},
        date={1984},
        ISSN={0022-1120},
     journal={J. Fluid Mech.},
      volume={144},
       pages={1\ndash 11},
}

\bib{zeldovich1980magnetic}{article}{
      author={Zeldovich, Yakov~Borisovich},
      author={Ruzmaikin, A.A.},
       title={Magnetic field of a conducting fluid in two-dimensional motion},
        date={1980},
     journal={Zhurnal Eksperimental'noi i Teoreticheskoi Fiziki},
      volume={78},
       pages={980\ndash 986},
}

\end{biblist}
\end{bibdiv}

\end{document}